\documentclass[a4paper,10pt]{article}

\usepackage{amsmath}
\usepackage{amssymb}
\usepackage{amsthm}
\usepackage{graphicx}
\usepackage{tikz}
\usepackage[nottoc,numbib]{tocbibind}
\usepackage[bw]{mcode}
\usepackage{epstopdf}
\usepackage[font=small,labelfont=bf]{caption}

\newtheorem{theorem}{Theorem}
\newtheorem{proposition}[theorem]{Proposition}
\newtheorem{conjecture}[theorem]{Conjecture}
\newtheorem{corollary}[theorem]{Corollary}

\begin{document}
\pagenumbering{arabic}
\title{A Comparison of Popular Point Configurations on $\mathbb{S}^2$}
\author{D.P. Hardin\footnote{Center for Constructive Approximation, Department of Mathematics, Vanderbilt University, Nashville, TN 37240, USA, email: doug.hardin@vanderbilt.edu, timothy.j.michaels@vanderbilt.edu, and edward.b.saff@vanderbilt.edu The research of the authors was supported, in part, by US-NSF grants DMS-1412428 and DMS-1516400.}, T. Michaels$^*$\footnote{The research of this author was completed as part of a Ph.D dissertation} and E.B. Saff$^*$}
\maketitle

\begin{abstract}
There are many ways to generate a set of nodes on the sphere for use in a variety of problems in numerical analysis. We present a survey of quickly generated point sets on $\mathbb{S}^2$, examine their equidistribution properties, separation, covering, and mesh ratio constants and present a new point set, equal area icosahedral points, with low mesh ratio. We analyze numerically the leading order asymptotics for the Riesz and logarithmic potential energy for these configurations with total points $N<50,000$ and present some new conjectures.

\


\

\end{abstract}

\newpage

\section{Introduction}
\subsection{Overview}
Distributing points on $\mathbb{S}^2:=\left\{x\in\mathbb{R}^3 : \ |x|=1\right\}$ is a classical problem arising in many settings in numerical analysis and approximation theory, in particular the study of radial basis functions, quadrature and polynomial interpolation, Quasi-Monte Carlo methods for graphics applications, finite element methods for PDE's, cosmic microwave background radiation modeling, crystallography and viral morphology, to name a few. The goal of this paper is to survey some widely used algorithms for the generation of spherical node sets. We will restrict our descriptions to ``popular" point sets, most of which can be generated ``reasonably fast." Namely, we study
\begin{itemize}
\item Fibonacci and generalized spiral nodes
\item projections of low discrepancy nodes from the unit square
\item zonal equal area nodes and HEALPix nodes
\item polygonal nodes such as icosahedral, cubed sphere, and octahedral nodes
\item minimal energy nodes 
\item maximal determinant nodes
\item random nodes and
\item ``mesh icosahedral equal area nodes." 
\end{itemize}
The last is a new point set devised to have many desirable properties. For each of the above configurations, we provide illustrations, and analyze several of their properties. We focus our attention primarily on
\begin{itemize}
\item equidistribution
\item separation
\item covering
\item quasi-uniformity and
\item Riesz potential energy.
\end{itemize}
For each property we provide numerical calculations, tables, and comparisons, and in some cases we prove theoretical bounds on the mesh ratio. Section 3 is devoted to asymptotic comparisons of various potential energies. We do not consider quadrature of the point sets; however, such a comparison for several of the configurations we describe here can be found in \cite{QMC}. We now formally introduce the properties we will be studying and, in Section 2, describe the point sets themselves. We leave the technical proofs to Section 4, and at the end of the document, we provide resources for Matlab source codes to generate the point sets.

\subsection{Definitions and Properties}

For low error numerical integration with respect to uniform surface area measure (\cite{ABD12} and \cite{KuipNied}) as well as in digitizing $\mathbb{S}^2$ for computer graphics purposes (\cite{signal} and \cite{FibInt}), it is important for any spherical configuration to have an approximately uniform distribution. A sequence $\left\{\omega_N\right\}_{N=1}^{\infty}$ of spherical point sets with $\omega_N$ having cardinality $N$ is called \textit{equidistributed} if the sequence of normalized counting measures, 
\[\nu_N(A):=\frac{1}{N}|A\cap\omega_N|,\ \ \ \ \ \ A\ \textup{Borel set},\]
 associated with the $\omega_N$'s converges in the weak-star sense to $\sigma$, the normalized surface area measure on $\mathbb{S}^2$, as $N\rightarrow \infty$. That is, for all continuous functions $f$ on $\mathbb{S}^2$,
 \[\lim_{N\to\infty}\int_{\mathbb{S}^2}f\,\text{d}\nu_N = \int_{\mathbb{S}^2}f\,\text{d}\sigma.\]
  An equivalent definition is that the $L_\infty$-spherical cap discrepancy
\[D_C(\omega_N):=\sup_{V\subset \mathbb{S}^2}\bigg|\frac{|V\cap \omega_N|}{N}-\sigma(V)\bigg|\to 0 ,\ \ \ \ \ \ \ N\to\infty,\]
where the supremum is taken over all spherical caps $V\subset \mathbb{S}^2$. 

For the study of local statistics, separation and covering properties play an important role. The \textit{separation} of a configuration $\omega_N\subset \mathbb{S}^2$ is

\[\delta(\omega_N):=\min_{\substack{x,y\in\omega_N\\ x\neq y}} |x-y|, \]
and a sequence of spherical $N$-point configurations is said to be \textit{well-separated} if for some $c>0$ and all $N\geq 2$,
\[\delta(\omega_N)\geq cN^{-1/2}.\]
The \textit{covering radius} of $\omega_N$ with respect to $\mathbb{S}^2$ is defined to be 

\[\eta(\omega_N):= \max_{y\in \mathbb{S}^2} \min_{x\in \omega_N} |x-y|,\]
and a sequence of spherical $N$-point configurations is a \textit{good-covering} if for some $C>0$ and all $N\geq 2$,
\[\eta(\omega_N)\leq CN^{-1/2}.\]
A sequence of configurations $\left \{\omega_N\right \}_{N=2}^\infty$ is said to be \textit{quasi-uniform} if the sequence

\[\bigg\{\gamma (\omega_N) := \frac{\eta (\omega_N)}{\delta (\omega_N)}\bigg\}_{N\geq 2}\]
is bounded as $N\rightarrow \infty.$ The quantity $\gamma(\omega_N)$ is called the \textit{mesh ratio} of $\omega_N$. Note that some authors define the mesh ratio as $2\gamma(\omega_N)$. A sequence of $N$-point configurations is quasi-uniform if it is well-separated and a good-covering. We remark that equidistribution does not imply quasi-uniformity or vice versa. In applications involving radial basis functions, ``1-bit" sensing, and finite element methods (\cite{RBF2}, \cite{RBF1}, \cite{RezSaff}, and \cite{FEM}), there is interest in precise bounds on $D_C(\omega_N)$, $\delta(\omega_N),\eta(\omega_N)$, and $\gamma(\omega_N)$. A trivial lower bound is $\gamma(\omega_N)\geq 1/2$ for any configuration. Asymptotically, as proved in \cite{BondHS}, 

\[\gamma(\omega_N) \geq \frac{1}{2\cos\pi/5} + o(1) = \frac{\sqrt{5}-1}{2} +o(1), \ \ \ \ N \rightarrow \infty,\]
for any sequence of configurations $\left \{\omega_N\right \}_{N=2}^\infty \subset \mathbb{S}^2$.

We also evaluate the potential energy of our point sets. The problem of minimizing point energies on the sphere dates to at least the beginning of the 20th century when Thomson put forth a model of the ground state configurations of electrons in \cite{Thoms}. Given a lower-semicontinuous, symmetric kernel $K:\mathbb{S}^2\times\mathbb{S}^2\rightarrow(-\infty,\infty]$, and a spherical configuration  $\omega_N\subset\mathbb{S}^2$, the $K$-energy of $\omega_N$ is defined to be

\begin{equation}
\text{E}_K(\omega_N):=\sum_{\substack{x,y\in\omega_N\\ x\neq y}} K(x,y).
\label{energy_def}
\end{equation}

\noindent The infimum of $\text{E}_K(\omega_N)$ over all $N$-point configurations on $\mathbb{S}^2$ is attained and is denoted by $\mathcal{E}_K(N)$. We will restrict our attention to the class of Riesz kernels defined by

\[\ \ \ \ \ \ \ \ \ \ \ 
K_s(x,y) = \frac{1}{|x-y|^s},\ \ \ \ \ \ \ \ s > 0\]
\[
K_{\text{log}}(x,y) = \text{log}\frac{1}{|x-y|},\]
\[\ \ \ \ \ \ \ \ \ \ \ 
K_s(x,y) = -|x-y|^{-s},\ \ \ \ \ \ \ \ s < 0.\]
For brevity, the energy and minimal energy quantities for the Riesz $s$-kernel and log kernel will be denoted by $\text{E}_s(\omega_N)$, $\text{E}_{\text{log}}(\omega_N)$, $\mathcal{E}_s(N)$, and $\mathcal{E}_{\text{log}}(N)$ respectively. Determining an exact minimal configuration for a fixed $N$ and $s$ is a highly nonlinear optimization problem. In practice, gradient descent and Newton methods are used to arrive at approximate global minima \cite{Calef}; however, there is substantial interest in generating nearly optimal points more quickly~\cite{Smale}.

The \textit{Voronoi cell} of a point $x\in\omega_N\subset\mathbb{S}^2$ is the spherical polygon
\[V_x(\omega_N):= \left\{y\in\mathbb{S}^2:\ |y-x| \leq |y-z|, \forall \ z\in\omega_N\setminus \left\{x\right\}\right\}.\] 
The \textit{Voronoi decomposition} of a configuration is
\[V(\omega_N):=\left\{V_x(\omega_N)\right\}_{x\in\omega_N}.\] It has been observed that the Voronoi cell decomposition of nearly optimal energy configurations appears to consist primarily of nearly regular spherical hexagons mixed with ``scars" of spherical heptagons and pentagons as shown in Figure \ref{MinEnVoronoi}. It has been conjectured that for $s>2$, the dominant term in the asymptotic expansion of $\mathcal{E}_s(N)$ is related to the Epstein-Zeta function for the hexagonal lattice (see Section 3). 

\begin{figure}[h!]
\centering
\includegraphics{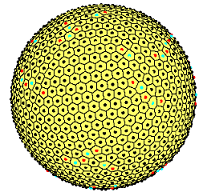}
\caption{Voronoi decomposition of approximately minimal energy nodes for $N=1000$ and $s=2$. Cells with black points are hexagonal, but not necessarily regular. Cells with red points are pentagonal and cells with cyan points are heptagonal (see online for color).}
\label{MinEnVoronoi}
\end{figure}

 We conclude this section by introducing some secondary properties of spherical configurations that we will use:
A partition $P_N := \left\{W_i\right\}_{i=1}^N$ of $\mathbb{S}^2$ into $N$ cells whose pairwise intersections have $\sigma$-measure $0$ is \textit{equal area} if $\sigma(W_i) = 1/N$ for all $1\leq i\leq N$. A sequence of partitions $\left \{P_N\right \}_{N=2}^\infty$ of $\mathbb{S}^2$ such that each $P_N$ has $N$ cells is \textit{diameter bounded} if there are constants $c,C > 0$ such that for all $N\in \mathbb{N}$ and for every cell $W^N_i\in P_N$,

\begin{equation}
cN^{-1/2}\leq \text{diam }W^N_i \leq CN^{-1/2},
\label{diambdd}
\end{equation}
\noindent where diam$(A):=\sup_{x,y\in A}|x-y|$. We will call a sequence of partitions of $\mathbb{S}^2$ \textit{asymptotically equal area} if 
\begin{equation}
\lim_{N\rightarrow\infty}N\max_{1\leq i\leq N} \sigma(W_i^{N}) = \lim_{N\rightarrow\infty}N\min_{1\leq i\leq N}\sigma(W_i^{N}) = 1,
\label{asympeq}
\end{equation}
\noindent and a sequence of spherical configurations $\left\{\omega_N\right\}_{N=1}^{\infty}$ will be said to be asymptotically equal area if its sequence of Voronoi decompositions is asymptotically equal area.

\section{Point Sets on $\mathbb{S}^2$}

\noindent\textbf{Generalized Spiral Points}

\

A spherical spiral on $\mathbb{S}^2$ is a path in spherical coordinates of the form 
\[r = 1,\ \ \ \ \ \theta = L\phi, \ \ \ \ \   0 \leq \phi \leq \pi,\]where $\phi$ denotes the polar angle and $\theta$ the azimuth.
Modifying a construction by Rakhmanov, Saff, and Zhou \cite{RSZ94}, Bauer \cite{Bauer} defines a sequence of $N$ points lying on a generating spherical spiral, $S_N$: 
\begin{equation}
L = \sqrt{N\pi}, \ \ \ h_k = 1 - \frac{2k-1}{N}, \ \ \ \phi_k = \cos^{-1}(h_k), \ \ \ \theta_k = L\phi_k, \ \ \ k=1,...,N.
\end{equation}

\noindent The slope $L$ is chosen such that for large $N$, the distance between adjacent points on the same level of the curve is similar to the distance between adjacent levels which differ by $2\pi$ in $\theta$. Indeed, the geodesic spacing between turns of the spiral is given by $2\pi/L = \sqrt{4\pi/N}$. Meanwhile, the total arc length is 

\[T = \int_{S_N} \sqrt{d\phi^2+d\theta^2\sin^2\phi} = \int_{0}^{\pi} \sqrt{1+L^2\sin^2\phi}\ d\phi = 2\sqrt{1+L^2}\textup{E}\big(L/\sqrt{1+L^2}\big),\]
where E$(\cdot)$ is the complete elliptic integral of the second kind. For large $N$,$\ \ T\approx~2L$, and the spiral is divided into nearly equal length segments of approximately $2L/N =\sqrt{4\pi/N}$. We refer to these points as the \textit{generalized spiral points}.
\      
\begin{theorem}
\label{spiral distribution}
 The sequence $\left\{\omega_N\right\}_{N=1}^{\infty}$ of generalized spiral point configurations is equidistributed on $\mathbb{S}^2$, quasi-uniform, and has the following asymptotic separation property:
 \begin{equation}
 \lim_{N\to \infty} \sqrt{N}\delta(\omega_N) = \sqrt{8-4\sqrt{3}\cos(\sqrt{2\pi}(1-\sqrt{3}))}\approx 3.131948....
 \label{spiralsep}
 \end{equation}
\end{theorem} 
 
 \begin{figure}
 \centering
 \includegraphics{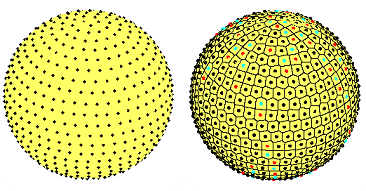}
 \caption{Plot of $N=700$ generalized spiral points and their Voronoi decomposition.}
 \label{spiral}
 \end{figure}
 
As shown in the proof of Theorem \ref{spiral distribution}, the Voronoi cells of $\omega_n$ are asymptotically equal area, but do not approach regular hexagons. Indeed, a typical decomposition is shown in Figure \ref{spiral}. A comparison of the mesh ratios for several values of $N$ is shown in Table \ref{tab:spiral}. Numerically, the mesh ratio appears to converge to 0.8099....

\begin{table}[h!]
\begin{center}
\caption{Mesh Ratios for Generalized Spiral Nodes}

\begin{tabular}{|c|c||c|c||c|c|}
\hline
$N$ & $\gamma(\omega_N)$ & $N$ & $\gamma(\omega_N)$& $N$ & $\gamma(\omega_N)$ \\
\hline\hline
10 & 0.897131& 400& 0.816007 & 20000& 0.809510\\
20 & 0.827821& 500& 0.810128 & 30000& 0.809629\\
30 & 0.814383& 1000& 0.805465& 40000& 0.809689\\
40 & 0.826281& 2000& 0.806411& 50000& 0.809725\\
50 & 0.834799& 3000& 0.807510& 100000& 0.809797\\
100& 0.803901& 4000& 0.808077& 200000& 0.809832\\
200& 0.806020& 5000& 0.808435& 300000& 0.809844\\
300& 0.809226& 10000& 0.809151& 500000& 0.809854\\

\hline
\end{tabular}
\label{tab:spiral}
\end{center}
\end{table}
\vspace{2cm}

\noindent\textbf{Fibonacci Nodes}

\

Another set of spiral points is modeled after nodes appearing in nature such as the seed distribution on the head of a sunflower or a pine cone, a phenomenon known as spiral phyllotaxis \cite{Dixon}. Coxeter \cite{Coxeter} demonstrated these arrangements are fundamentally related to the Fibonacci sequence, $\{F_k\} = \{1,1,2,3,5,8,13,...\}$ and the golden ratio $\varphi = (1+\sqrt{5})/2$. There are two similar definitions of the spherical point set in the literature. Both are defined as lattices on the square $[0,1)^2$ and then mapped to the sphere by the Lambert cylindrical equal area projection, denoted by $\Lambda$. In Cartesian coordinates, $\Lambda$ is defined by

\[\Lambda(x,y):=(\sqrt{1-(2y-1)^2}\cos 2\pi x,\sqrt{1-(2y-1)^2}\sin 2\pi x,2y-1)\]
and, in spherical coordinates, by 

\[\Lambda(x,y):=(\cos^{-1}(2y-1),2\pi x) = (\phi,\theta).\]
Define a rational lattice on $[0,1)^2$, with total points $F_k$ by

\begin{equation}
\widetilde{\omega}_{F_k} :=  \bigg(\bigg\{\frac{iF_{k-1}}{F_k}\bigg\},\frac{i}{F_k}\bigg),\ \ \ 0\leq i\leq F_k,
\label{FibLattice}
\end{equation}
\noindent where $\left\{x\right\} = x - \lfloor{x} \rfloor $ denotes the fractional part of $x$. On the other hand, an irrational lattice can be formed similarly for all values of total points $N$ by replacing $F_{k-1}/F_{k}$ in (\ref{FibLattice}) by $\lim_{k\rightarrow\infty}F_{k-1}/F_{k} = \varphi^{-1}$:

\[\widetilde{\omega}_{N} :=  \bigg(\left\{i\varphi^{-1}\right\},\frac{i}{N}\bigg),\ \ \ 0\leq i\leq N.\]

\begin{figure}
\centering
\includegraphics{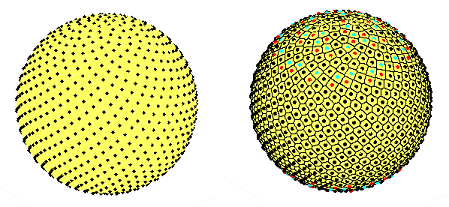}
\caption{Fibonacci nodes for $N=1001$ as defined by Swinbank and Purser and corresponding Voronoi decomposition. The visible spirals at each level are generated by a sequence of points whose index increases arithmetically by a Fibonacci number.}
\end{figure}

Swinbank and Purser \cite{SwinPurs} define a spherical point set for all odd integers $2N+1$ symmetrically across the equator derived from the irrational lattice with points shifted a half step away from the poles:

\[\ \theta_i = 2\pi i \varphi^{-1}, \ \ \sin \phi_i = \frac{2i}{2N+1},\ \ \ -N\leq i \leq N, \ \ \ -\pi/2 \leq \phi_i \leq \pi/2.\]
Denote $\omega_{2N+1}$ as the configuration generated above. Whereas for large $N$, the generalized spiral points tend towards flattening out and partitioning the sphere into distinct regions of latitude, the Fibonacci points maintain visible clockwise and counterclockwise spirals as $N$ grows. Labeling the points of $\omega_{2N+1}$ by increasing latitude, the dominant spirals emanating from $x_i\in \omega_{2N+1}$ are formed by the sequence $\left\{x_{i+jF_k}\right\}$ for some Fibonacci number $F_k$ and $j=\cdots,-2,-1,0,1,2,\cdots$.

The Fibonacci points derived from the rational lattice are studied by Aistleitner et al \cite{ABD12} and Bilyk et al \cite{Bilyk}  for discrepancy estimates. In \cite{ABD12}, the spherical cap discrepancy of the points $\left\{F_k\right\}$ is bounded by

\begin{figure}

\noindent\makebox[\textwidth][c]{
\resizebox{10cm}{!}{%
\begin{tikzpicture}

\node at (3,-1.1) [font=\bf] {y$_i$};
\node at (1.5,-.15) [font=\bf] {b$_2$};
\node at (5.5,-.35) [font=\bf] {b$_1$};
\node at (3.45,0.7) [font=\bf] {b$_5$};
\node at (3.7,0) [font=\bf] {b$_3$};
\node at (2.3,0.6) [font=\bf] {b$_4$};
\node at (2.7,2.3) [font=\bf] {b$_6$};


\draw (-4,-3)--(9,-3);
\draw (-4,-2.8)--(-4,-3.2);
\node at (-4,-3.5) {$0$};
\draw (9,-2.8)--(9,-3.2);
\node at (9,-3.5) {$2\pi$};

\draw (2.96,-2)--(7.63,-2);
\node at (5.5, -2.35) {$2\pi\varphi^{-1}$};
\draw (2.96,-1.85)--(2.96,-2.15);
\draw (7.63,-1.85)--(7.63,-2.15);

    \clip (-4,-2) rectangle (9cm,6cm);
    \pgftransformcm{.8506}{.5257}{-.6243}{.9056}{\pgfpoint{0cm}{0cm}}

   \coordinate (Bzero) at (2,-2);        
   \coordinate (Bone) at (6,-4);
   \coordinate (Btwo) at (0,0);
   \coordinate (Bthree) at (4,-2);
   \coordinate (Bfour) at (2,0);
   \coordinate (Bfive) at (4,0);
   \coordinate (Bsix) at (4,2);

    \foreach \x in {-6,-5,...,6}{
      \foreach \y in {-6,-5,...,6}{ 
        \node[draw,circle,inner sep=1.5pt] at (2*\x,2*\y) {};
         
      }
    }
    
        \draw [ultra thick,-latex,black] (Bzero)
            -- (Bone) node [above, font=\bf] {y$_{i+1}$};
		\draw [ultra thick,-latex,black] (Bzero)
		    -- (Btwo) node [above, font=\bf] {y$_{i+2}$};
		\draw [ultra thick,-latex,black] (Bzero)
		    -- (Bthree) node [above, font=\bf] {y$_{i+3}$};
		\draw [ultra thick,-latex,black] (Bzero)
		    -- (Bfour) node [above, font=\bf] {y$_{i+5}$};
		\draw [ultra thick,-latex,black] (Bzero)
		    -- (Bfive) node [above, font=\bf] {y$_{i+8}$};
		\draw [ultra thick,-latex,black] (Bzero)
		    -- (Bsix) node [above, font=\bf] {y$_{i+13}$};
\end{tikzpicture}}}
\caption{Irrational lattice points on the square with labeled basis vectors from \cite{SwinPurs} around a node $\textbf{y}_i$. The configuration is approximately a rotation of the rectangular lattice by the angle $\tan^{-1}(\varphi)$. Each line of points along a basis vector is mapped to a spiral under the Lambert projection, $\Lambda$. The basis vector $\textbf{b}_k$ such that $\Lambda(\textbf{b}_k)$ has the shortest length forms the visible dominant spirals emanating from $\Lambda(\textbf{y}_i)$. These shortest length vectors are determined by the total number of points and zone number~$z$.}
\label{fig:Fibbasis}
\end{figure}
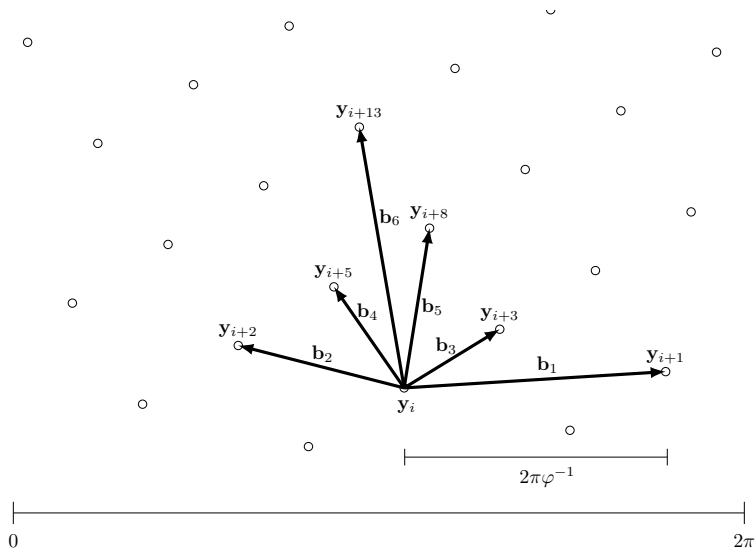

\[
D_C(\Lambda(\widetilde{\omega}_{F_k})) \leq \left\{
\begin{array}{lr}
44\sqrt{(2/F_k)}& \ \textup{if \textit{k} is odd,}\ \\
44\sqrt{(8/F_k)}& \ \textup{if \textit{k} is even.}
\end{array} \right.
\]

\noindent Numerical experiments in \cite{ABD12} suggest that in fact,

\[ D_C(\Lambda(\widetilde{\omega}_{F_k})) = O\bigg(\frac{(\log F_k)^c}{F_k^{3/4}}\bigg),\ \ \ \ k\rightarrow\infty \ \ \ \text{for some}\ \  1/2\leq c \leq 1.\]
which is optimal up to a log power \cite{Beck}. Both sequences of Fibonacci configurations are equidistributed. However, since the Swinbank and Purser nodes are defined for more values of total points, we will take these to be the Fibonacci sets moving forward. In \cite{SwinPurs}, these points are also numerically shown to be asymptotically equal area.

Analyzing $\omega_{2N+1}$ as a shifted irrational lattice mapped by the Lambert projection helps to visualize the underlying spiral structure. Define a system of basis vectors 
\[ \textbf{b}_k = \Lambda^{-1}(x_{i+F_k})-\Lambda^{-1}(x_i), \ \ \ \ k=1,2,3...,\]
 which are independent of base point $x_i$. This is illustrated in Figure \ref{fig:Fibbasis}. Emanating from each point $x_i$, the line of points $\left\{x_i+m\textbf{b}_k\right\}_{m=\dots ,-1,0,1,\cdots}$ is mapped to a spiral on $\mathbb{S}^2$ under the Lambert projection. Like the Fibonacci sequence, the basis vectors satisfy

\[\textbf{b}_{k+1} = \textbf{b}_k + \textbf{b}_{k-1}.\]
On the sphere, the basis vectors in terms of the local Cartesian coordinate system at a point $(\phi_i,\theta_i)\in\omega_{2N+1}$ have the form

\begin{equation}
\textbf{c}_{k,i} = \bigg( (-1)^k2\pi\cos\phi_i\varphi^{-k}, \frac{2F_k}{(2N+1)\cos\phi_i}\bigg).
\label{Fib basis}
\end{equation}
For a fixed latitude $\phi$ and total number of points $2N+1$, the zone number $z$ is defined by

\[\varphi^{2z} = (2N+1)\pi\sqrt{5}\cos^2\phi.\]
Letting $d = \sqrt{4\pi/(\sqrt{5}(2N+1))}$ and using the fact that for large $k$, $F_k \approx \varphi^k/\sqrt{5}$ equation (\ref{Fib basis}) can be rewritten

\begin{equation}
\label{Fib basis2}
\textbf{c}_{k,i} \approx d((-1)^k\varphi^{z-k}, \varphi^{k-z}).
\end{equation}

\begin{figure}
\centering
\includegraphics{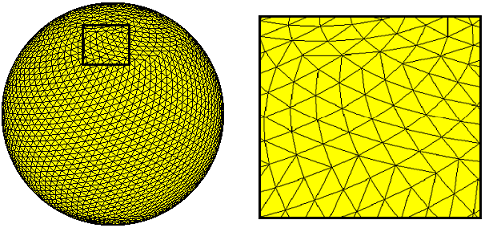}
\caption{Triangulation of $N=3001$ nodes viewed at a slightly different angle. The enlarged box demonstrates where zone number $z$ changes with changing $\phi$. The sudden shift occurs at $z=k\pm1/2$. Along this latitude, $\textbf{c}_{k,i}$ and $\textbf{c}_{k\pm 1,i}$ have equal lengths. For $k-1/2< z< k+1/2$, the triangulation consists of basis vectors $\textbf{c}_{k,i}$ and $\textbf{c}_{k+1,i}$ as proven in \cite{SwinPurs}.}
\label{fig:Fib tri}
\end{figure}

\noindent For latitudes where $k-1/2\leq z\leq k+1/2$, $\textbf{c}_{k,i}$ has the minimum length of the basis vectors around $x_i$ and forms the dominant spiral at those latitudes. As shown in \cite{SwinPurs}, $|\textbf{c}_{k,i}|$ is also the smallest distance between points near these latitudes. Thus, the Delaunay triangulation \cite{Delaunay} of $\omega_{2N+1}$ is composed of $\textbf{c}_{k,i}, \textbf{c}_{k-1,i}$, and $\textbf{c}_{k+1,i}$ when $k-1/2\leq z \leq k+1/2$. This is shown in Figure \ref{fig:Fib tri} and allows us to prove quasi-uniformity.

\begin{proposition}
\label{Fib mesh ratio}
The sequence of Fibonacci configurations is quasi-uniform.
\end{proposition}

Numerically, the minimal separation appears to occur at the pole with value $|x_1-x_4| = |x_{2N+1}-x_{2N-2}|$ and the largest hole appears to occur in the triangles covering the poles, $\triangle x_2,x_3,x_5$ and $\triangle x_{2N},x_{2N-1},x_{2N-3}$. In a straightforward computation, it can be shown that 
\[\lim_{N\to\infty} \sqrt{2N+1}|x_1-x_4| = \sqrt{16-\sqrt{112}\cos(6\pi\varphi^{-1})} = 3.09207...\]
and the circumradius $r$ of the polar triangles satisfies
\[\lim_{N\to\infty}\sqrt{2N+1}\,r = 2.72812....\]
As shown in Table \ref{tab:Fib}, the mesh ratios for Fibonacci nodes appear to converge quickly to this ratio $\approx 0.882298$.

\begin{table}[h!]
\begin{center}
\caption{Mesh Ratios for Fibonacci Nodes}
\begin{tabular}{|c|c||c|c||c|c|}
\hline
$N$ & $\gamma(\omega_N)$ & $N$ & $\gamma(\omega_N)$& $N$ & $\gamma(\omega_N)$ \\
\hline\hline
11 & 0.859197& 401& 0.881897 & 20001& 0.882289\\
21 & 0.872632& 501& 0.881978 & 30001& 0.882292\\
31 & 0.876251& 1001& 0.882139& 40001& 0.882293\\
41 & 0.877909& 2001& 0.882218& 50001& 0.882294\\
51 & 0.878857& 3001& 0.882244& 100001& 0.882296\\
101& 0.880646& 4001& 0.882258& 200001& 0.882297\\
201& 0.881489& 5001& 0.882266& 300001& 0.882297\\
301& 0.881762& 10001& 0.882282& 500001& 0.882297\\

\hline
\end{tabular}

\label{tab:Fib}
\end{center}
\end{table}

\

\noindent \textbf{Low Discrepancy Nodes}

\

Another approach for distributing points on the sphere is to minimize a suitable notion of discrepancy, such as spherical cap, $L_p$, or generalized discrepancy (cf. \cite{BD15} and \cite{CF97}). A low spherical cap discrepancy sequence $\left\{\omega_N\right\}_{N=2}^{\infty}$ satisfies \cite{Beck}
\begin{equation}
\frac{a}{N^{3/4}}\leq D_C(\omega_N)\leq A\frac{\sqrt{\log N}}{N^{3/4}}, \ \ \ \ \ \ N\geq 2,
\end{equation}
\noindent for some $a,A>0$. Low discrepancy point sets are used in Quasi-Monte Carlo methods for numerical integration and also in graphics applications in \cite{WongLH}.
One method for generating spherical nodes is to first distribute points on the square~$[0,1)^2$ with low planar discrepancy \cite{Nied92}, i.e. for some $A>0$ 
\begin{equation}
D(\omega_N):=\sup_{R}\bigg|\frac{R\cap\omega_N}{N}-\sigma(R)\bigg|\leq A\frac{\log N}{N},\ \ \ \ \omega_N\subset[0,1)^2,
\label{disc}
\end{equation}  
\noindent where the supremum $R$ is taken over all rectangles with sides parallel to the axes. These sequences are then mapped to the sphere via the Lambert projection. While the $\log N/N$ term in (\ref{disc}) is optimal on the plane \cite{Schmidt}, it is an open problem whether the images of these sequences have optimal order spherical cap discrepancy. There are several such node distributions in the literature (cf. \cite{ABD12} and \cite{CF97}), but as their properties are similar, we only consider the following one. 

\

\noindent \underline{Hammersley Nodes}

\

For an integer $p \geq 2$, the $p$-adic van der Corput sequence is defined by

\[x^{(p)}_k = \frac{a_0}{p} + \dots + \frac{a_r}{p^{r+1}},\ \ \ \  \text{where}\ \ \ \  k = a_0 + \dots +a_rp^r,\ \ \ a_i\in\left\{0,1\right\}.\]
The Hammerlsy node set on the square (\cite{CF97}, \cite{Nied92}, and \cite{WongLH}) is given by $x_k :=  x^{(2)}_k$ and $y_k := \frac{2k-1}{2N}$. The $N$ point spherical Hammersley node set is given by $\{\Lambda(2\pi x_k,1-2y_k)\}^N_{k=1}$. The configuration for $N=1000$ is given in Figure \ref{Hamfig}. The discrepancy of the planar Hammersley nodes is known from Niederreiter~\cite{Nied92} to satisfy (\ref{disc}) . The sequence of Hammersley configurations is equidistributed; however it is not well-separated or quasi-uniform. This makes the nodes poor candidates for energy, as shown in Section 3. Their Voronoi decompositions also exhibit no discernible geometric patterns. 

\begin{figure}
\centering
\includegraphics{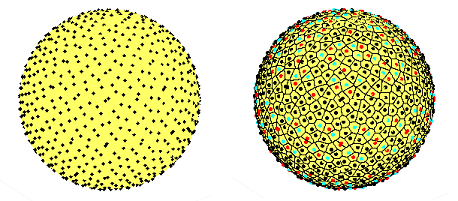}
\caption{Hammersley nodes for $N=1000$ and corresponding Voronoi decomposition.}
\label{Hamfig}
\end{figure}

\

\noindent\textbf{Equal Area Partitions}

\

Another class of point sets are those derived from equal area partitions of the sphere.

\vspace{2cm}

\noindent\underline{Zonal Equal Area Nodes}:

\

Rakhmanov et al \cite{RSZ94} construct a diameter bounded, equal area partition of $\mathbb{S}^2$ into rectilinear cells of the form
\[R([\tau_\phi,\nu_\phi]\times [\tau_\theta,\nu_\theta]):= \left\{(\phi,\theta)\in\mathbb{S}^2\ :\  \tau_\phi\leq\phi\leq\nu_\phi, \tau_\theta\leq\theta\leq\nu_\theta\right\}.\]
The cells are grouped by regions of equal latitude called collars that have the form $R([\tau_\phi,\nu_\phi]\times [0,2\pi])$. The cells are defined such that $\nu_\phi-\tau_\phi$ and $\nu_\theta-\tau_\theta$ approximate $\sqrt{4\pi/N}$ as $N$ grows. This ensures the correct order of the diameter bound. The cells are defined in the following way.

\

1. \textit{Determine the latitudes of the polar caps.} The first two cells are taken to be the polar caps of radius $\phi_c = \cos^{-1}(1-2/N).$

\

2. \textit{Determine the ideal collar angle and ideal number of collars.} The ideal angle between two collars is 
\[\delta_I := \sqrt{4\pi/N}.\]
 The ideal number of collars between the polar caps, all of which have angle $\delta_I$, is 
 \[n_I:=\frac{\pi-2\phi_c}{\delta_I}.\]

\begin{figure}
\label{ZonalPartition}
\centering
\includegraphics{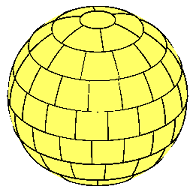}
\caption{Zonal equal area partition of the sphere into 100 cells. The algorithm does not determine a unique partition as the collars can rotate past each other. Image created with the help of Paul Leopardi's equal area partitioning toolbox available at eqsp.sourceforge.net.}
\end{figure}

3. \textit{Determine the actual number of collars, $n$.} If $N = 2,$ then $n := 0$. Otherwise, 
\[n := \max\left\{1,\text{round}(n_I)\right\}.\] 

4. \textit{Create a list of the ideal number of cells in each collar.} The ``fitting" collar angle is 
\[\delta_F := \frac{n_I}{n}\delta_I = \frac{\pi-2\phi_c}{n}.\]
Label the collars $\left\{C_j\right\}_{j=1}^{n+2}$ southward with the North polar cap as $C_1$ and the South polar cap as $C_{n+2}$. The area $A_j$ of collar $C_j$ can be written as the difference of polar cap areas:
\[A_j = 2\pi(\cos(\phi_c+(j-2)\delta_F)-\cos(\phi_c+(j-1)\delta_F)).\]
Thus the ideal number of cells $y_{j,I}$ in each collar $C_j$, $j\in\left\{2,\dots,n+1\right\}$, is given by
\[y_{j,I} = \frac{4\pi A_j}{N}.\]

5. \textit{Create a list of the actual number of cells in each collar.} We apply a cumulative rounding procedure. Letting $y_j$ be the number of cells in $C_j$, define the sequences $y$ and $a$ by $a_1:=0$, $y_1:=1$, and for $j\in \left\{2,\dots,n+1\right\}$:
\[y_j:=\text{round}(y_{j,I}+a_{j-1}),\ \ \ \ \ \ a_j:=\sum_{k=1}^{j}y_k-y_{k,I}.\]

6. \textit{Create a list of latitudes $\phi_j$ of each collar and partition each collar into cells.} We define $\phi_j$ as follows: $\phi_0 = 0$, $\phi_{n+2} = \pi$ and for $j\in \left\{1,\dots,n+1\right\}$, 
\[\phi_j = \cos^{-1}(1-\frac{2}{N}\sum_{k=1}^j y_k).\]
Thus the North polar cap of radius $\phi_j$ has normalized area $\sum_{k=1}^j y_k/N$, and $C_j:=R([\phi_{j-1},\phi_j]\times[0,2\pi])$. 

\

7. \textit{Partition each collar into cells.} $C_j$ has $y_j$ equal cells  \[\left\{R([\phi_{j-1},\phi_j]\times[\theta_j+k\frac{y_j}{2\pi},\theta_j+(k+1)\frac{y_j}{2\pi}])\right\}_{k=0}^{y_j-1},\] 
where $\theta_j\in[0,2\pi)$ can be chosen to be any starting angle. Note that because $\theta_j$ are chosen independently, the equal area partition determined by the algorithm is not unique. Indeed, the collars can rotate past each other without affecting the diameter bound or equal area property of the partition. Because the choice of the $\theta_j$'s does not strongly affect the other properties with which we are concerned, we will take them to be random.

\begin{figure}
\label{zonal}
\centering
\includegraphics{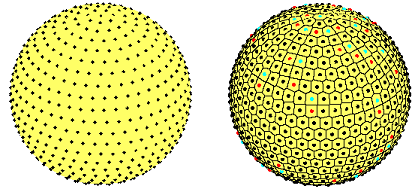}
\caption{Zonal equal area points for $N=700$ and corresponding Voronoi decomposition. The configurations have a similar structure to the generalized spiral points in Figure \ref{spiral}.}
\end{figure}

The point set $\omega_N$ is defined to be the centers of the cells of the rectilinear partition. As proved by Zhou \cite{Zhou}, the cells are diameter bounded from above by $7/\sqrt{N}$; however, numerical experiments from Leopardi in \cite{Leopardi} suggest the bound to be $6.5/\sqrt{N}$. For large $N$, the zonal equal area configurations look very similar to the generalized spiral configurations. Namely they exhibit iso-latitudinal rings with separation between adjacent points equal to separation between rings and a random longitudinal shift between points in adjacent rings. As shown in Section~3, the energy computations for both point sets are nearly identical. 

\begin{proposition}
The sequence of zonal equal area configurations is equidistributed and quasi-uniform.
\label{zonal quasi}
\end{proposition}

The above construction was modified by Bondarenko et al \cite{BondVia} to create a partition with geodesic boundaries for the creation of well-separated spherical designs. More details can be found in \cite{BondVia}. Table \ref{tab:Zonal} gives a comparison of the mesh ratios of the zonal points.

\begin{table}[h!]
\begin{center}
\caption{Mesh Ratios for Zonal Equal Area Nodes}
\begin{tabular}{|c|c||c|c||c|c|}
\hline
$N$ & $\gamma(\omega_N)$ & $N$ & $\gamma(\omega_N)$& $N$ & $\gamma(\omega_N)$ \\
\hline\hline
10 & 0.711934& 400& 0.769527 & 20000& 0.758100\\
20 & 0.790937& 500& 0.766808 & 30000& 0.758069\\
30 & 0.788546& 1000& 0.765356& 40000& 0.756793\\
40 & 0.843385& 2000& 0.764631& 50000& 0.756785\\
50 & 0.790252& 3000& 0.758645& 100000& 0.756770\\
100& 0.761296& 4000& 0.756510& 200000& 0.756762\\
200& 0.764846& 5000& 0.764217& 300000& 0.758015\\
300& 0.763188& 10000& 0.758192& 500000& 0.756757\\

\hline
\end{tabular}
\label{tab:Zonal}
\end{center}
\end{table}

\

\noindent\underline{HEALPix Nodes}

\

\begin{figure}
\centering
\includegraphics{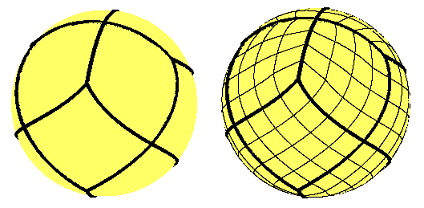}
\caption{Base tessellation of the sphere into 12 equal area pixels. A finer mesh is created by dividing each base pixel into a $k\times k$ grid of equal area subpixels of the same shape. The HEALPix points are taken to be the centers of the pixels.}
\label{HEALPix}
\end{figure}

 Developed by NASA for fast data analysis of the cosmic microwave background (CMB), the Hierarchical Equal Area iso-Latitude Pixelization (HEALPix) was designed to have three properties essential for computational efficiency in discretizing functions on $\mathbb{S}^2$ and processing large amounts of data \cite{HEALPix}:
 \begin{enumerate}
\item The sphere is hierarchically tessellated into curvilinear quadrilaterals.
\item The pixelization is an equal area partition of $\mathbb{S}^2$.
\item The point sets are distributed along fixed lines of latitude.
\end{enumerate}

To create the partition of $\mathbb{S}^2$, the authors in \cite{HEALPix} first divide the sphere into~12 equal area, four sided pixels defined by the following boundaries:

\[|\cos\phi|>\frac{2}{3}, \ \ \ \ \  \theta=m\frac{\pi}{2}, \ \ \ m = 0,1,2,3\]

\[\cos\phi=\frac{-2-4m}{3}+\frac{8\theta}{3\pi},\ \ \ \ \  \frac{m\pi}{2}\leq\theta\leq\frac{(m+1)\pi}{2}, \ \ \ \ \ m=0,1,2,3\]

\[\cos\phi=\frac{2-4m}{3}-\frac{8\theta}{3\pi},\ \ \ \ \  \frac{-(m+1)\pi}{2}\leq\theta\leq\frac{-m\pi}{2}, \ \ \ \ \ m=0,1,2,3.\]
The base tessellation is shown in Figure \ref{HEALPix}. For a given $k\in \mathbb{N}$, each pixel is partitioned into a $k \times k$ grid of sub-pixels of the same shape and equal area. The HEALPix point sets are taken to be the centers of these pixels. 

On the polar regions $|\cos\phi|>2/3$, the points are distributed along $k$ iso-latitudinal rings, indexed by $i$, each with $4i$ equally spaced points, indexed by~$j$:

\[|\cos\phi_i| = 1 - \frac{i^2}{3k^2}, \ \ \ \ \ \ \ \ \ \theta_j = \frac{\pi}{2i}\bigg(j-\frac{1}{2}\bigg).\]
On the equatorial region, there are $2k-1$ iso-latitudinal rings, each with $4k$ points. The rings are indexed by $k\leq |i|\leq 2k$ and the points by $1\leq j \leq 4k$:

\[|\cos\phi_i| = \frac{4}{3} - \frac{2i}{3k},\]

\[\theta_j = \frac{\pi}{2k}\bigg(j-\frac{s}{2}\bigg), \ \ \ \ \ s = (i-k+1)\mod{2}.\]
The index $s$ describes the phase shift between rings. This gives a configuration of size $N=12k^2$. The point sets are hierarchical along the subsequence $k=3^m$. Holho\c s and Ro\c sca \cite{HR15} have shown that the HEALPix points can be obtained as the image of points on a certain convex polyherdon under an area preserving mapping to the sphere.

\begin{figure}
\centering
\includegraphics{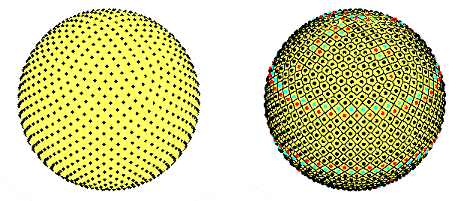}
\caption{HEALPix nodes and Voronoi decomposition for $N=1200, k=10$.}
\end{figure}

\begin{proposition}
The sequence of HEALPix configurations is equidistributed and quasi-uniform.
\label{HEALPix quasi}
\end{proposition}

\noindent Numerically, the mesh ratio appears to be bounded by 1, as shown in Table \ref{tab:HEALPix}.

\begin{table}[h!]
\begin{center}
\caption{Mesh Ratios for HEALPix nodes}
\begin{tabular}{|c|c|c||c|c|c||c|c|c|}
\hline
$k$ & $N$ & $\gamma(\omega_N)$ &$k$ & $N$ & $\gamma(\omega_N)$ &$k$ & $N$ & $\gamma(\omega_N)$\\
\hline\hline
1 & 12 & 0.864783 & 9& 972& 0.965950& 45& 24300& 0.992956\\
2& 48& 0.862243& 10& 1200& 0.969599& 50& 30000& 0.993648\\
3& 108& 0.909698 & 15& 2700& 0.979371& 60& 43200& 0.994701\\
4& 192& 0.929080& 20& 4800& 0.984328& 70& 58800& 0.995456\\
5& 300& 0.940016& 25& 7500& 0.987365& 80& 76800& 0.996020\\
6& 432& 0.951047& 30& 10800& 0.989509& 90& 97200& 0.996455\\
7& 588& 0.957584& 35& 14700& 0.990959& 100& 120000& 0.996807\\
8& 768& 0.961782& 40& 19200& 0.992082& 150& 270000& 0.997867\\

\hline
\end{tabular}
\label{tab:HEALPix}
\end{center}
\end{table}
 
 \
 
\noindent\textbf{Polyhedral Nodes and Area Preserving Maps}

\

Another class of point sets are those derived from subdividing regular polyhedra and applying radial projection: $\Pi(\textbf{x}) = \textbf{x}/\|\textbf{x}\|$ or an equal area projection. These node sets are used in finite element methods to give low error solutions to boundary value problems. See, for instance, \cite{GHWIcos} and \cite{NTLCube}.

\

\noindent\underline{Radial Icosahedral Nodes} 

\

This point set, as described in \cite{Teanby} and \cite{Wright} is formed by overlaying a regular triangular lattice onto each face of a regular icosahedron of circumradius $1$ and edge length $a = \csc(2\pi/5)$. Given $k\in \mathbb{N}$, for each vertex $v$, divide two adjacent edges emanating from $v$ into basis vectors of length $a/k$. For the face $\mathcal{F}$ determined by these edges and vertex, the icosahedral point set $\widetilde{\omega_{N_k}}$ on $\mathcal{F}$ is taken to be the set of lattice points generated by these basis vectors restricted to $\mathcal{F}$. The spherical points are $\omega_{N_k}:=\Pi(\widetilde{\omega_{N_k}})$. These node sets are defined for total points $N = 10k^2+2$ and hierarchical along the subsequence $k_j = k_02^j$ for any $k_0\in \mathbb{N}$.

The sequence of icosahedral configurations $\left\{\widetilde{\omega_{N_k}}\right\}_{k=1}^{\infty}$ is equidistributed. However, because radial projection is not area preserving, the sequence of spherical configurations is not equidistributed. Density is higher towards the vertices of the icosahedron and lower towards the center of the faces where the areal distortion of $\Pi$ is greatest. The Voronoi decomposition of $\omega_{N_k}$ is composed of twelve regular pentagons with all other cells regular hexagons of varying size as illustrated in Figure \ref{RadIcos}.

\begin{proposition}
\label{icosahedral distribution}
The sequence of radial icosahedral configurations is quasi-uniform.
\label{rad icos quasi}
\end{proposition}

\noindent Numerically, the mesh ratio appears to be bounded by $0.86$, as shown in Table~\ref{tab:RadIcos}.

 \begin{figure}
  \centering
  \includegraphics{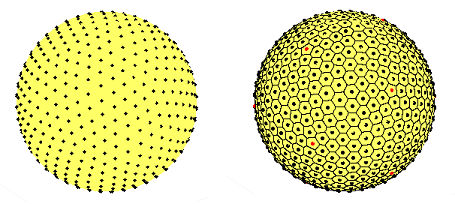}
  \caption{Radial icosahedral odes $N=642$. The Voronoi decomposition is composed of regular hexagons of varying size and 12 pentagons at the vertices of the icosahedron.}
  \label{RadIcos}
  \end{figure}

\begin{table}[h!]
\begin{center}
\caption{Mesh Ratios for Radial Icosahedral Nodes}
\begin{tabular}{|c|c|c|c|c|c|}
\hline
$k$ & $N$ & $\gamma(\omega_N)$& $k$ & $N$ & $\gamma(\omega_N)$\\
\hline\hline
1& 12&  0.620429 & 20& 4002& 0.830750\\
2& 42& 0.667597 & 30& 9002& 0.838066\\
3& 92& 0.684698 & 40& 16002& 0.842358\\
4& 162& 0.745348& 50& 25002& 0.844697\\
5& 252& 0.765157 & 60 & 36002& 0.846156\\
6& 362& 0.769854& 70 & 49002 & 0.847376\\
7& 492& 0.789179& 100& 100002 & 0.849390\\
10& 1002& 0.808024& 150& 225002 & 0.850941\\
15& 2252& 0.821504& 200& 400002 & 0.851745\\
\hline
\end{tabular}
\label{tab:RadIcos}
\end{center}
\end{table}

\vspace{2cm}

\noindent\underline{Cubed Sphere Nodes}

\

A similar method as above can be applied to the cube \cite{NTLCube}. A square $k\times k$ grid is placed on each face of the cube and radially projected to the sphere. A typical point set is shown in Figure \ref{Cubed Sphere}. The configurations are defined for $N = 6k^2 -12k +8$ and are hierarchical along the subsequence $k=k_02^m$. By an argument similar to that in Proposition \ref{icosahedral distribution}, the limiting distribution is not uniform, but the sequence of configurations is quasi-uniform. Numerically, the mesh ratio seems to quickly converge to 1, as shown in Table \ref{tab:Cube}.

\begin{figure}
\centering
\includegraphics{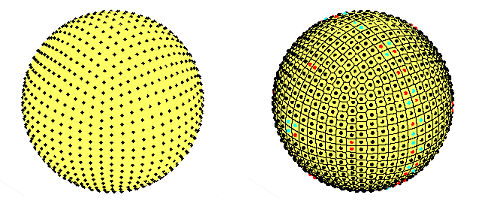}
\caption{Plot of $N=1016$ Cubed Sphere Points and Voronoi Decomposition. The Voronoi cells tend towards regular hexagons near the vertices of the cube. Towards the middle of each face they resemble a square lattice.}
\label{Cubed Sphere}
\end{figure}

\begin{table}[h!]
\begin{center}
\caption{Mesh Ratios for Cubed Sphere Points}
\begin{tabular}{|c|c|c||c|c|c||c|c|c|}
\hline
$k$ & $N$ & $\gamma(\omega_N)$ &$k$ & $N$ & $\gamma(\omega_N)$ &$k$ & $N$ & $\gamma(\omega_N)$\\
\hline\hline
2& 8& 0.827329& 10& 488& 0.996846& 50& 14408& 0.999893\\
3& 26& 0.794265& 15& 1178& 0.994025& 60& 20888& 0.999926\\
4& 56& 0.972885& 20& 2168& 0.999289& 70& 28568& 0.999946\\
5& 98& 0.933655& 25& 3458& 0.997954& 80& 37448& 0.999959\\
6& 152& 0.989913& 30& 5048& 0.999695& 90& 47528& 0.999968\\
7& 218& 0.968757& 35& 6938& 0.998979& 100& 58808& 0.999974\\
8& 296& 0.994805& 40& 9128& 0.999831& 150& 133208& 0.999988\\
9& 386& 0.982046& 45& 11618& 0.999390 & 200& 237608& 0.999994\\
\hline
\end{tabular}
\label{tab:Cube}
\end{center}
\end{table}

\noindent\underline{Octahedral Points}

\

 Unlike in the previous examples, the octahedral points, described by Holho\c s and Ro\c sca \cite{Oct}, are derived from an area preserving map $\mathcal{U}$ from the regular octahedron $\mathbb{K}$ of edge length $L = \sqrt{2\pi}/\sqrt[4]{3}$ and surface area $4\pi$ to $\mathbb{S}^2$. Let $\mathcal{U}_x, \mathcal{U}_y, $ and $\mathcal{U}_z$ denote the $x$,$y$, and $z$ components of $\mathcal{U}$ respectively. For $(X,Y,Z)\in\mathbb{K}$,

\[ \mathcal{U}_z = \frac{2Z}{L^2}(\sqrt{2}L-|Z|),\]

\[\mathcal{U}_x = \text{sgn}(X)\sqrt{1-\mathcal{U}_z^2}\cos\frac{\pi |Y|}{2(|X|+|Y|)},\]

\[\mathcal{U}_y = \text{sgn}(Y)\sqrt{1-\mathcal{U}_z^2}\sin\frac{\pi |Y|}{2(|X|+|Y|)}.\]

To produce a spherical point set, the authors form a partition $P_k$ of $k^2$ triangles on each face of the octahedron in the same manner as the radial icosahedral points and obtain an equal area spherical partition $\mathcal{P}_k = \mathcal{U}(P_k)$. The point sets~$\omega_{N_k}$ are taken to be the vertices of the triangles of $\mathcal{P}_k$. For a given~$k$, there are~$8k^2$ triangles and $N=4k^2+2$ points.

\begin{figure}
\centering
\includegraphics{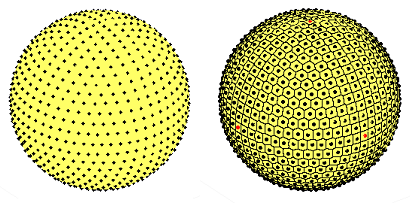}
\caption{Equal area octahedral points for $k=15$ and $N=902$. The Voronoi decomposition of the octahedral points is composed of hexagons and 8 squares at the vertices of the octahedron. The hexagons approach regularity towards the center of the faces and deform along the edges.}
\end{figure}

The octahedral configurations have similar properties to the HEALPix node sets. They are iso-latitudinal and hierarchical along the subsequence $k=k_02^m$. The sequence of configurations is equidistributed, and in addition, the authors in \cite{Oct} compute a diameter bound for any triangular region $\mathcal{T}$ of $\mathcal{P}_k$ to be
\[\textup{diam}\ \mathcal{T}\leq \frac{2\sqrt{4+\pi^2}}{\sqrt{8k^2}}\approx \frac{7.448}{\sqrt{8k^2}}.\] Following their proof of this bound, we can calculate a lower bound on the separation and an upper bound on the mesh norm. 

\begin{theorem}
\label{octahedral distribution}
The sequence of equal area octahedral configurations is quasi-uniform with 
\begin{equation}
\gamma(\omega_{4k^2+2})\leq \frac{1}{4} \sqrt{\frac{4+\pi^2}{2-(k+1)^2/k^2}}\rightarrow \frac{\sqrt{4+\pi^2}}{4}\approx 0.931048..., \ \ \ \ k\rightarrow\infty.
\label {Octmesh}
\end{equation}
\label{oct quasi}
\end{theorem}

\noindent This bound seems to be near optimal. As shown in Table \ref{tab:Oct}, the mesh ratio grows to at least $0.9235$.

\

\

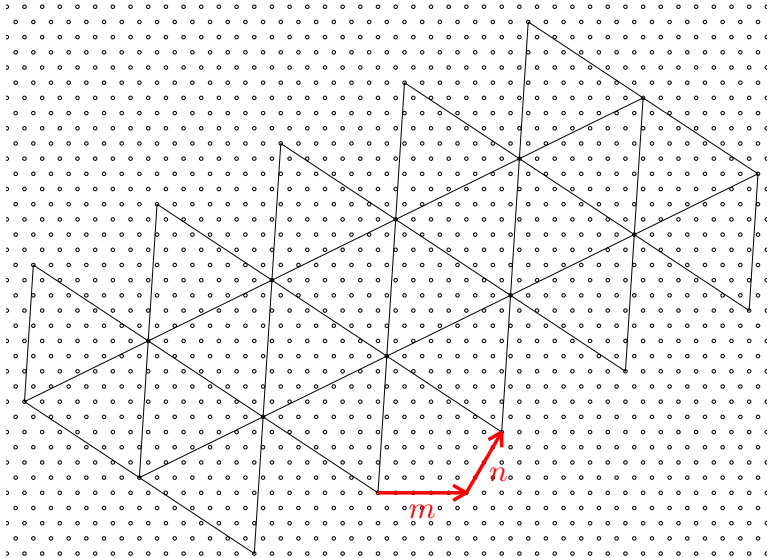
\begin{figure}
\noindent\makebox[\textwidth][c]{
\resizebox{10cm}{!}{%
\begin{tikzpicture}

    \clip (-4.5,-2) rectangle (17cm,14cm);
    \pgftransformcm{1}{0}{.5}{.866}{\pgfpoint{0cm}{0cm}}

  \foreach \x in {-25,-24,...,38}{
  \foreach \y in {-5,-4,...,40}{ 
  \node[draw,circle,inner sep=1pt, thick] at (1/2*\x,1/2*\y) {};
  }
  }

        \draw [thick] (-1,.5)--(11.5,10.5);
        \draw [thick] (-5.5,3)--(7,13);
        
         \draw [thick] (-5.5,3)--(3.5,-2);
         \draw [thick] (-7.5,7.5)--(6,0);
         \draw [thick] (-5,9.5)--(8.5,2);
         \draw [thick] (-2.5,11.5)--(11,4);
         \draw [thick] (0,13.5)--(13.5,6);
         
         \draw [thick] (-7.5,7.5)--(-5.5,3);
         \draw [thick] (-5,9.5)--(-1,.5);
         \draw [thick] (-2.5,11.5)--(3.5,-2);
         \draw [thick] (0,13.5)--(6,0);
         \draw [thick] (2.5,15.5)--(8.5,2);
         \draw [thick] (7,13)--(11,4);
         \draw [thick] (11.5,10.5)--(13.5,6);
         \draw [thick] (2.5,15.5)--(11.5,10.5);
         
         \draw [red, line width = .1cm] (6,0)--(8.5,0) node [pos = 0.5, scale = 2.5, red, very thick, below]{$m$};
         \draw [red, line width = .1cm] (8.5,0)--(8.5,2) node [pos = 0.3, scale = 2.5, red, very thick, right]{$n$};
         \draw [red, line width = .1cm] (8,.25)--(8.5,0);
         \draw [red, line width = .1cm] (8.25,-.25)--(8.5,0);
         \draw [red, line width = .1cm] (8.25,1.75)--(8.5,2);
         \draw [red, line width = .1cm] (8.75,1.5)--(8.5,2);
         
\end{tikzpicture}}}
\caption{Planar icosahedral mesh for $(m,n) = (5,4)$.}
\label{IcosMesh}
\end{figure}

\begin{table}[h!]
\begin{center}
\caption{Mesh Ratios for Octahedral Nodes}
\begin{tabular}{|c|c|c||c|c|c||c|c|c|}
\hline
$k$ & $N$ & $\gamma(\omega_N)$ &$k$ & $N$ & $\gamma(\omega_N)$ &$k$ & $N$ & $\gamma(\omega_N)$\\
\hline\hline
1 & 6 & 0.675511 & 9& 326& 0.873510& 60& 14402& 0.905758\\
2& 18& 0.872884& 10& 402& 0.875606& 70& 19602& 0.908047\\
3& 38& 0.854610& 15& 902& 0.882510& 80& 25602& 0.909875\\
4& 66& 0.856329& 20& 1602& 0.886310& 90& 32402& 0.911382\\
5& 102& 0.860536& 25& 2502& 0.888702& 100& 40002& 0.912644\\
6& 146& 0.864599& 30& 3602& 0.892884& 200& 160002& 0.919218\\
7& 198& 0.868095& 40& 6402& 0.898762& 300& 360002& 0.921947\\
8& 258& 0.871036& 50& 10002& 0.902784& 400& 640002& 0.923503\\

\hline
\end{tabular}
\label{tab:Oct}
\end{center}
\end{table}

\noindent\underline{Mesh Icosahedral Equal Area Points}

\

Improvements to the radially projected icosahedral points have been put forth by Song et al \cite{SongKim} and Tegmark \cite{Tegmark}. Here, we introduce two other improvements to these points to create new configurations. First, we generalize the icosahedral lattice structure to create configurations of more possible numbers of total points. Due to a method of Caspar and Klug \cite{CK62} derived during their investigation of the construction of viruses, we define a triangular lattice on a regular icosahedron with total points
\[N = 10(m^2+mn+n^2) + 2, \ \ \ \ \ \ \ \ (m,n)\in \mathbb{N}\times\mathbb{N}\setminus(0,0).\]

Consider the planar triangular lattice generated by $e_1 = (1,0)$ and $e_2 = (1/2,\sqrt{3}/2)$. For a given $(m,n)$, let $e_{m,n} = me_1+ne_2$ and it's rotation by~$\pi/3$ be basis vectors for an unfolded icosahedron superimposed on the lattice. This is illustrated in Figure \ref{IcosMesh}. Folding the icosahedron results in a triangular lattice~$\widetilde{\omega_{N}}$ on each face. Due to rotational symmetry of the lattice, the resulting configuration is independent of how the icosahedron is unfolded. The subsequence~$(m,0)$ produces the lattice for the radial icosahedral nodes.

\begin{figure}
\noindent\makebox[\textwidth][c]{
\begin{tikzpicture}

\draw (-0.5,-0.5)--(6.5,-0.5)--(3,5.562)--cycle;

\draw (15,0) arc (0:60:6cm);

\draw (12,5.196) arc(120:180:6cm);

\draw (9,0) arc(240:300:6cm);

\draw  (5.5,3.5) [->,very thick] arc(110:70:5);

\draw [red, very thick](3,1.732)--(3,0) node [pos = 0.5, red, left]{$h$};

\draw [red, very thick](3,0)--(4.5,0) node [pos = 0.5, red, below]{$w$};

\draw [red, very thick](4.5,0)--(3,1.732);

\draw (3,.3)--(3.3,.3)--(3.3,0);

\draw (4.5,0)--(5.73,0);

\draw [dashed, very thick](3,0)--(3,-0.5);

\draw [dashed, very thick](3,1.732)--(6.5,-0.5);

\node at (4.5,0) {\textbullet};

\node at (4.5,-0.2) {$p$};

\node at (3.5,.6) [red] {$\mathcal{T}$};

\draw (12,1.732) [red, very thick] arc (150:188.5:3cm);
\draw (11.63,-0.2) [dashed, thick] arc (188.5:200:3cm);

\draw (15,0) [dashed, thick] arc (30:90:3.48cm);

\draw (11.63,-0.2) [red,very thick] arc (265:287:5cm);
\draw (13.5,0) arc (287:301.25:5cm);

\draw (13.5,0) [red, very thick] arc (30:52:6cm);

\draw (11.9,1.5) arc (269:287:1cm);
\draw (14.5,0.7) arc (150:169:1cm);

\draw (11.61,0.1)--(11.93,0.1)--(11.93,-.2);

\node at (13.5,0) {\textbullet};
\node at (3,1.732) {\textbullet};
\node at (12,1.732) {\textbullet};
\node at (3,0) {\textbullet};
\node at (11.63,-0.2) {\textbullet};
\node at (5.73,0) {\textbullet};
\node at (14.65,0.48) {\textbullet};

\node at (7.25,4.2) {\Large $\Phi$};
\node at (13.5,-0.3) {$\Phi(p)$};
\node at (12.5,0.5) [red] {$\Phi(\mathcal{T})$};
\node at (3,2) {$\mathcal{O}$};
\node at (12,2) {$\Phi(\mathcal{O})$};
\node at (12.05,1.3) {$\lambda$};
\node at (2.7,0) {$p_{\bar{A}}$};
\node at (11,-0.2) {$\Phi(p_{\bar{A}})$};
\node at (5.7,-0.2) {$p_{\bar{B}}$};
\node at (14.2,0.6) {$\psi$};
\node at (14.35,0) {$\Phi(p_{\bar{B}})$};

\end{tikzpicture}}
\caption{Illustration of area preserving map $\Phi$ defined piecewise on each triangle bounded by two altitudes of a face with relevant variables labeled.}
\label{IcosMap}
\end{figure}
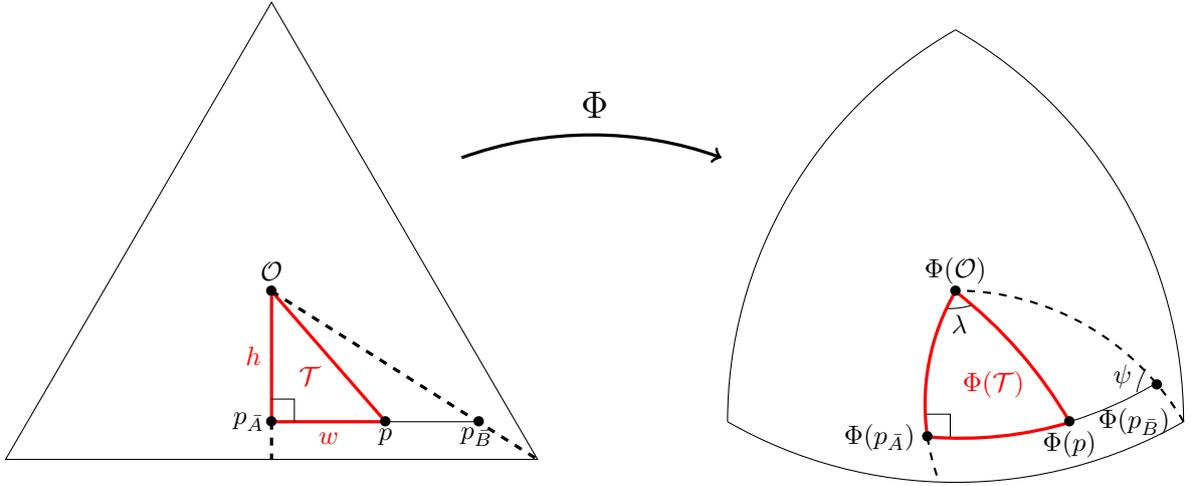

Secondly, we derive an area preserving map $\Phi$ from the regular icosahedron~$\mathcal{I}$ of edge length $L = \sqrt{4\pi}/\sqrt[4]{75}$, circumradius $r = L\sin2\pi/5$, and surface area~$4\pi$ using the technique presented by Snyder \cite{Snyder} for the truncated icosahedron. We define $\Phi$ piecewise by dividing each face $\mathcal{F}\subset \mathcal{I}$ into the six triangles $\mathcal{R}_i$ partitioned by the altitudes of $\mathcal{F}$:

\

1. Parametrize each point $p\in\mathcal{R}_i$ by $h$ and $w$ as labeled in Figure \ref{IcosMap}. If $\bar{A}$ is the side of $\mathcal{R}_i$ of length $L/2\sqrt{3}$, then $w$ is the distance from $p$ to $\bar{A}$ and $h$ is the distance of $p_{\bar{A}}:= proj(p,\bar{A})$ to $\mathcal{O}$, the center of $\mathcal{F}$.

\

2. Let $\bar{B}$ be the side of $\mathcal{R}_i$ of length $L/\sqrt{3}$ and $p_{\bar{B}}$ be the intersection of the line $\overline{p_{\bar{A}}p}$ with $\bar{B}$. For the triangle $\mathcal{S} = \triangle p_{\bar{A}}p_{\bar{B}}\mathcal{O}$, find $\psi$ as in Figure \ref{IcosMap} and spherical right triangle $\Phi(\mathcal{S})$ such that Area$(\mathcal{S}) = \sigma(\Phi(\mathcal{S})) = \psi - \pi/6$. Thus,

\[\psi = \frac{h^2\sqrt{3}}{2}+\frac{\pi}{6}.\]

3. $\Phi(p)$ will lie on the great circle $\Phi(\overline{p_{\bar{A}}p_{\bar{B}}})$. Letting $\mathcal{T} = \triangle p_{\bar{A}},p,\mathcal{O}$, find $\lambda$ as in Figure \ref{IcosMap} such that Area$(\mathcal{T}) = \sigma(\Phi(\mathcal{T}))$. By the spherical law of cosines, 

\[\sigma(\Phi(\mathcal{T})) = \lambda + \cos^{-1}\bigg(\frac{2\sin\lambda\cos\psi}{\sqrt{3}}\bigg)-\frac{\pi}{2}.\]
Thus

\begin{figure}
\centering
\includegraphics{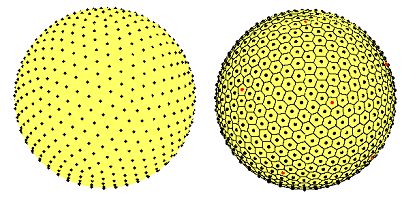}
\caption{Equal area mesh icosahedral nodes for $(m,n) = (7,2)$ and total points $N=672$. The Voronoi decomposition forms a spherical Goldberg polyhedron.}
\label{meshicos}
\end{figure}

\noindent \begin{table}[h!]

\begin{center}
\caption{Mesh Ratios for Equal Area Icosahedral Nodes}
\resizebox{\columnwidth}{!}{
 \begin{tabular}{|c|c|c||c|c|c||c|c|c|}
\hline
$(m,n)$ & $N$ & $\gamma(\omega_N)$ &$(m,n)$ & $N$ & $\gamma(\omega_N)$ &$(m,n)$ & $N$ & $\gamma(\omega_N)$\\
\hline\hline
(1,0) & 12 & 0.620429 & (7,1)& 572& 0.688031& (37,27)& 30972& 0.730733\\
(1,1)& 32 & 0.617964& (7,5)& 1092& 0.707058& (40,33)& 40092& 0.732326\\
(2,0)& 42& 0.667598&(12,4)& 2082& 0.706688& (42,40)& 50442& 0.732529\\
(2,1)& 72& 0.657081&(16,3)& 3132& 0.704123& (65,50)& 99752& 0.733719\\
(3,0)& 92& 0.659610&(16,7)& 4172& 0.717067& (90,75)& 204752& 0.735013\\
(3,1)& 132& 0.668227&(19,6)& 5112& 0.712681& (100,100)& 300002& 0.735397\\
(4,1)& 212& 0.680153& (19,18)& 10272& 0.726243& (131,100)& 402612& 0.735592\\
(5,2)& 392& 0.687368& (31,21)& 20532& 0.728761& (145,115)& 509252& 0.735965\\

\hline

\end{tabular}}
\label{tab:MeshIcos}
\end{center}

\end{table}

\[\tan\lambda = \frac{\sin(\frac{hw}{2})}{\cos(\frac{hw}{2})-\frac{2\cos\psi}{\sqrt{3}}}.\]

4. Transform $(\psi,\lambda)$ into spherical coordinates.

\

\noindent The map $\Phi$ is extended to $\mathcal{I}$ by rotations and reflections. This defines the unique azimuthal equal area projection from $\mathcal{I}$ onto $\mathbb{S}^2$ \cite{Mapbook}. The spherical configurations are $\omega_N :=\Phi(\widetilde{\omega_{N}})$. A typical point set is shown in Figure \ref{meshicos}. The Voronoi cells are almost regular hexagons with 12 pentagonal cells at the vertices of $\mathcal{I}$, and the Voronoi decomposition forms a spherical Goldberg Polyhedron \cite{Goldberg}. To implement the points in Matlab, we derive explicit formulas on a triangular face. More details are provided in the companion paper \cite{Michaels}.

Unlike the radial icosahedral points, the sequence of equal area configurations is equidistributed. Regarding quasi-uniformity, the following is proved in~\cite{Michaels}.
\begin{proposition}
The sequence of equal area icosahedral configurations are quasi-uniform with
\[\gamma(\omega_N)\leq 0.798....\] 
\end{proposition}

\newpage

As shown in Table \ref{tab:MeshIcos}, the mesh ratios appears to stay below $0.736$. These are the lowest mesh ratios of all point sets discussed.

\

\noindent \textbf{Coulomb Points and Log Energy Points (Elliptic Fekete Points)}

\

\begin{figure}
\centering
\includegraphics{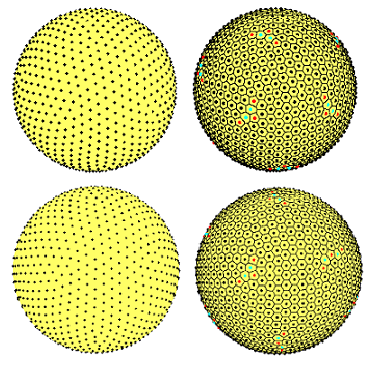}
\caption{Coulomb points (top) and log points (bottom) and their Voronoi decomposition for $N=1024$. The structure is very similar for both.}
\label{Coulomb}
\end{figure}

For $s=1$, Riesz $s$-energy minimization is the classic Thomson problem for the Coulomb potential \cite{Thoms}. The sequence of minimal Coulomb energy configurations is known to be equidistributed, well-separated, and quasi-uniform \cite{Dahlberg}. However, no explicit bound is known for the mesh ratio. The Voronoi decomposition of these cells, as shown in Figure \ref{Coulomb}, primarily consists of close to regular hexagons with heptagons and pentagons forming scars along the sphere. For relatively small $N$, the scars grow out from the 12 vertices of the icosahedron like dislocations in a crystal due to displacement deformities. For $N >~5,000$, the scars become less fixed, spreading across the sphere. For an in-depth discussion of the scarring behavior, see Bowick et al \cite{Bow2002} and \cite{Bow2009}.

The log energy points are minimizers of the Riesz logarithmic potential. The sequence of log energy configurations is known to be equidistributed and well-separated, but covering and quasi-uniformity is an open problem. As shown in Table \ref{tab:CoulLog} below, numerically the log energy points appear to be quasi-uniform. The best known lower bound on separation is due to Dragnev \cite{logsep}:

\[\delta(\omega^{\log}_N)\geq\frac{2}{\sqrt{N}},\ \ \ N\geq 2.\]
 Their geometric structure is very similar to the Coulomb points as shown in Figure \ref{Coulomb}. The energies of log and Coulomb points have the same asymptotic behavior in the dominant and second order term for many Riesz potentials. See Section 3 and Conjecture \ref{newconj}.

Generating these points is a highly nonlinear optimization problem. Unlike the configurations we have described up to now, they are not so quickly obtained. Table \ref{tab:CoulLog} displays the mesh ratios of near minimal Coulomb and log energy configurations. We remark that the sequence appears to have outliers at several values of $N$, such as $N = 20, 300,$ and $4096$. Points for $N<500$ and $N = k^2$, $k \leq 150$, were provided by Rob Womersley.

\begin{figure}
\centering
\includegraphics{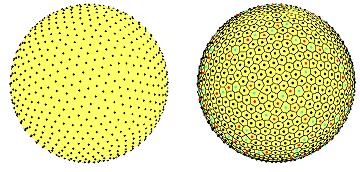}
\caption{Maximal determinant nodes for $N=961$. The Voronoi decomposition is primarily composed of regular hexagons with scarring features similar to the minimal energy points.}
\label{Maxdet}
\end{figure}

\begin{table}[h!]
\begin{center}
\caption{Mesh Ratios for Coulomb and Log Energy Points}
\begin{tabular}{|c|c|c||c|c|c|}
\hline
$N$ & $\gamma(\omega^{\text{log}}_N)$& $\gamma(\omega^{\text{Coul}}_N)$ & $N$ & $\gamma(\omega^{\text{log}}_N)$ & $\gamma(\omega^{\text{Coul}}_N)$ \\
\hline\hline
10 & 0.687401& 0.689279& 500&  0.757354& 0.755834\\
20 & 0.731613& 0.733265& 1024& 0.752122& 0.755770\\
30 & 0.695481& 0.692966& 2025& 0.761261& 0.766218\\
40 & 0.669531& 0.670842& 3025& 0.765075& 0.761661\\
50 & 0.661301& 0.656591& 4096& 0.770240& 0.765712\\
100& 0.695371& 0.694604& 5041& 0.753573& 0.758457\\
200& 0.662102& 0.658561& 10000& 0.762672& 0.761964\\
300& 0.740635& 0.730182& 15129& 0.762385& 0.763398\\
400& 0.650106& 0.647351& 22500& 0.773483& 0.767096\\

\hline

\end{tabular}
\label{tab:CoulLog}
\end{center}
\end{table}

\

\noindent\textbf{Maximal Determinant Nodes (Fekete Nodes)}

\

Other node sets used in polynomial interpolation and numerical integration on the sphere are the maximal determinant nodes. Let $\phi_1,...,\phi_{(n+1)^2}$ be a basis for the space  $\mathbb{P}_n(\mathbb{S}^2)$ of spherical polynomials of degree $\leq n$. The maximal determinant node set is the configuration $\omega_N := \omega_{(n+1)^2}\subset \mathbb{S}^2$ which maximizes

\[\textup{det}(\phi_i(x_j))_{i,j=1}^{(n+1)^2}\]
These points are independent of the choice of basis. The interpolatory cubature rule associated with the configuration $\omega_N$,

\[Q_n(f):=\sum_{j=1}^{N}w_jf(x_j),\]
is conjectured in \cite{Wom} to have all weights positive which is of interest in numerical integration. For more information about these points and their applications, see \cite{Reim94}, \cite{Reim99}, and \cite{Wom}. A typical node set is shown in Figure \ref{Maxdet}.

Like the minimal energy nodes, computing the maximal determinant nodes is a nonlinear optimization problem. The maximum is approximated by conjugate gradient and Newton-like methods on $\mathbb{S}^2$ \cite{Wom}. Nodes for $1\leq n\leq 165$ are available from \textit{http://web.maths.unsw.edu.au/\textasciitilde rsw/Sphere/Extremal/New/index.html}.

Berman et al \cite{BerBoucks} proved the sequence of maximal determinant configurations is equidistributed, while in \cite{Wom}, Sloan and Womersley proved it is quasi-uniform with 
\[\limsup_{N\to\infty}\gamma(\omega_N) < \frac{4j_0}{\pi}\approx 3.06195,\]
where $j_0$ is the smallest positive zero of the Bessel function of the first kind, $J_0$. As shown in Table \ref{tab:MaxDet}, the mesh ratio bound appears to be much lower though it is unclear whether or not $\lim_{N\to\infty}\gamma(\omega_N)$ exists.

\begin{figure}
\centering
\includegraphics{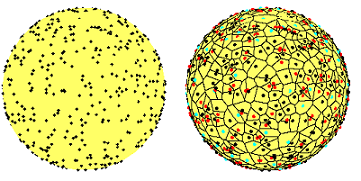}
\caption{Random points for $N=700$ and their Voronoi decomposition.}
\label{Random}
\end{figure}

\begin{table}[h!]
\begin{center}
\caption{Mesh Ratios for Maximal Determinant Nodes}
\begin{tabular}{|c|c|c|c|}
\hline
 $N$ & $\gamma(\omega_N)$&  $N$ & $\gamma(\omega_N)$\\
\hline\hline
9&  0.718884 & 625& 0.805608\\
16& 0.685587 & 1024& 0.840506\\
25& 0.768510 & 2025& 0.858874\\
36& 0.806140& 3025& 0.847347\\
49& 0.777490 & 4096& 0.859887\\
100& 0.708579& 4900 & 0.877990\\
225& 0.860728& 10201 & 0.859625\\
324& 0.799227& 15129 & 0.865695\\
400& 0.809172& 22500 & 0.881492\\
\hline
\end{tabular}
\label{tab:MaxDet}
\end{center}
\end{table}

\

\noindent \textbf{Random Points}

\

The final configurations we consider are random configurations $\omega_N^{rand}$ consisting of $N$ independent samples chosen with respect to surface area measure. Not surprisingly, these configurations do not have optimal order separation or covering and the sequence is not quasi-uniform. As proved in \cite{Randcov} and \cite{RezSaff} respectively,

\[ \lim_{N\rightarrow\infty}\mathbb{E}(\delta(\omega^{\text{rand}}_N))N = \sqrt{2\pi}, \ \ \ \ \ \lim_{N\rightarrow\infty}\mathbb{E}(\eta(\omega^{\text{rand}}_N)) \bigg(\frac{N}{\log N}\bigg)^{1/2}= 2. \]
Note that while the order of the separation of random points is off by factor of $N^{1/2}$, the covering is only off by a factor of (log$N)^{1/2}$. Figure \ref{Random} shows a realization of i.i.d. uniformly chosen random points on $\mathbb{S}^2$. 

\

\noindent \textbf{Summary of Properties}

\

The following tables compare some of the properties of the point sets described above. Table \ref{tab:comp1} compares which sequences are proven to be equidistributed and well-separated, for which values of $N$ the configurations are defined, and whether a subsequence is hierarchical. Table \ref{tab:comp2} compares which sequences are quasi-uniform and the numerically determined bounds for separation and mesh ratio constants.

\
\begin{table}[h!]
\begin{center}
\caption{Summary of Point Set Properties}
\begin{tabular}{|c|c|c|c|c|} 
 \hline
 Name & Defined for & Hier. & Equidist. & Separated \\
 \hline
 Gen Spiral & $N\geq 2$ & No & Yes & Yes \\ 
 Fibonacci & Odd $N$ & No & Yes & Yes  \\
 Hammersley & $N\geq 2$& No& Yes& No \\
 Zonal Eq. Area & $N \geq 2$ & No & Yes & Yes \\
 HEALPix & $12k^2$, & Subseq.& Yes & Yes \\
 Octahedral & $4k^2 + 2$& Subseq.& Yes& Yes  \\
 Radial Icos. & $10k^2+2$, & Subseq. & No & Yes \\
 Cubed Sphere & $6k^2-12k+8$& Subseq.& No& Yes \\
 Equal Area Icos. & $10(m^2+mn+n^2)+2$ & Subseq. & Yes & Yes \\
 Coulomb & $N\geq 2$ & No & Yes & Yes \\
 Log Energy & $N\geq 2$& No & Yes& Yes \\
 Max Det. & $(1+k)^2$ & No & Yes & Yes \\
 Random & $N\geq 2$& No & Yes & No\\
 \hline

\end{tabular}
\label{tab:comp1}
\end{center}

\end{table}

\begin{table}[h!]
\begin{center}
\caption{Comparison of Separation and Mesh Ratio Constants}
\begin{tabular}{|c|c|c|c|} 
 \hline
 Name & Quasi- & Numeric lower bound  & Numeric upper bound \\
  & Uniform& on $\liminf\delta(\omega_N)\sqrt{N}$  & on $\limsup\gamma(\omega_N)\sqrt{N}$ \\
  \hline
 Gen Spiral & Yes & 3.1319 & 0.8099  \\ 
 Fibonacci & Yes & 3.0921& 0.8823  \\
 Hammersley & No & N/A & N/A\\
 Zonal Eq. Area & Yes & 3.3222  & 0.7568  \\
 HEALPix & Yes & 2.8345 & 1.0000 \\
 Radial Icos. & Yes & 2.8363 & 0.8517\\
 Cubed Sphere & Yes & 2.7027  & 1.0000\\
 Octahedral & Yes & 2.8284 & 0.9235\\
 Equal Area Icos. & Yes & 3.1604 & 0.7360 \\
 Coulomb & Yes & 3.3794 & 0.7671\\
 Log Energy & Conj. & 3.3733 & 0.7735\\
 Max Det. & Yes & 3.1957 & 0.8900\\
 Random & No & N/A & N/A\\
 \hline

\end{tabular}

\label{tab:comp2}
\end{center}
\end{table}

\section{Asymptotic Energy Considerations}

We now turn our attention to the potential energy of the above configurations, in particular, $E_{\log}(\omega_N)$ and $E_s(\omega_N)$ for $s$ = -1,1,2, and 3. We are interested in the asymptotic behavior of the energies for above configurations as $N \rightarrow \infty$ and how well the leading terms in the expansion of their energies match the known or conjectured leading terms and coefficients in the minimal energy expansion. We begin with the following well known necessary condition for asymptotically optimal point sets \cite{MEBook}.

\begin{theorem}
Let  $\left\{\omega_N\right\}_{N=2}^\infty$ be a sequence of configurations that is asymptotically optimal with respect to logarithmic energy or Riesz $s$-energy for some $s>-2$, $s\neq 0$, i.e.,
 \[\lim_{N\rightarrow\infty}\frac{E_{s}(\omega_N)}{\mathcal{E}_{s}(N)} = 1 \ \ \ \ \ \ \ \  \textup{or} \ \ \ \ \ \ \ \ \lim_{N\rightarrow\infty}\frac{E_{\log}(\omega_N)}{\mathcal{E}_{\log}(N)} = 1.\]
 Then  $\left\{\omega_N\right\}_{N=2}^\infty$ is equidistributed.
\label{optimpliesequi}
\end{theorem}

We define the energy of a probability measure $\mu$ on $\mathbb{S}^2$ with respect to the logarithmic or Riesz $s$-potential as

\[\ \ \ \ \ \ \ \ \ \ \ \mathcal{I}_{s} [\mu]:=\ \int\int \frac{1}{|\textbf{x} - \textbf{y}|^s} \text{d}\mu(\textbf{x})\text{d}\mu(\textbf{y}),\ \ \ \ \ \ s\neq 0\]
\[\mathcal{I}_{\log} [\mu]:=\ \int\int \log\frac{1}{|\textbf{x} - \textbf{y}|} \text{d}\mu(\textbf{x})\text{d}\mu(\textbf{y}).\]
The normalized surface area measure $\sigma$ is the unique minimizer of $\mathcal{I}_{\log} [\mu]$ and $\mathcal{I}_{s} [\mu]$ for $-2<s<2$, $s\neq 0$, and 

\[\mathcal{I}_{\log} [\sigma] = \frac{1}{2}-\log2,\]
\[\mathcal{I}_s[\sigma] = \frac{2^{1-s}}{2-s}, \ \ \ \  -2<s<2.\]
However, for $s\geq 2$, $\mathcal{I}_{s} [\mu] = \infty$ for all $\mu$ supported on $\mathbb{S}^2$ (see for example \cite{MEBook}).

For random points, the expected value of the $s$-energy is easily computed:
\begin{equation}
\mathbb{E}[E_{s}(\omega_N^{\text{rand}})] = \mathcal{I}_{s}[\sigma](N(N-1)),\ \ \ \ \ -2<s<2.
\label{RandomEnergy}
\end{equation}
For $s\geq 2$, $\mathbb{E}[E_{s}(\omega_N^{\text{rand}})] = \infty$.

The Epstein-Zeta function for a lattice $\Lambda$ in $\mathbb{R}^2$ is given by
  
  \[\zeta_{\Lambda}(s):= \sum\limits_{0\neq\textbf{x}\in\Lambda} |\textbf{x}|^{-s},\ \ \ \ \ \ \textup{Re}\ s>2.\]
 Let $\Lambda_2$ be the regular triangular lattice in $\mathbb{R}^2$ generated by basis vectors $(1,0)$ and $(1/2,\sqrt{3}/2)$. It is known from number theory that $\zeta_{\Lambda_2}(s)$ admits the factorization

\begin{equation}
\zeta_{\Lambda_2}(s) = 6\zeta(s/2)\text{L}_{-3}(s/2),\ \ \ \ \ \ \text{Re}\ s>2,
\label{EpsZeta}
\end{equation}
where
 \[\text{L}_{-3}(s) := 1-\frac{1}{2^s}+\frac{1}{4^s}-\frac{1}{5^s}+\frac{1}{7^s}-\cdots,\ \ \ \ \ \ \text{Re}\ s>1,\]
 is a Dirichlet L-series. The right-hand side of (\ref{EpsZeta}) can be used to extend $\zeta_{\Lambda_2}(s)$ to $s$ with Re $s<2$. 
 
 \
 
\subsection{Logarithmic Potential}

The following asymptotic expansion is proven by Betermin and Sandier in \cite{Betermin}:
\begin{theorem}
\label{logexpansion}
There exists a constant $C\neq 0$, independent of $N$, such that 
  \[\mathcal{E}_{\log}(N) = \mathcal{I}_{\log} [\sigma] N^2-\frac{N\log N}{2} + CN + o(N), \ \ \ \ \ \ N\rightarrow \infty,\]
  \[ -0.22553754 \leq C\leq \widehat{C}:=2\log2 +\frac{1}{2}\log\frac{2}{3} +3\log\frac{\sqrt{\pi}}{\Gamma(1/3)} = -0.05560530... \ \ \ \ \ \ \]
\end{theorem}
\noindent The following extension of Theorem \ref{logexpansion} is conjectured by Brauchart et al in \cite{BHS12}:

\begin{conjecture} 
\label{logconjecture}
 \[\mathcal{E}_{\log}(N) = \mathcal{I}_{\log} [\sigma] N^2-\frac{N\log N}{2} + \widehat{C}N +D\log N + O(1), \ \ \ \ \ \ \ N\rightarrow \infty.\]

\end{conjecture}
\noindent Beltran \cite{Bel15} provides a partial converse to Theorem \ref{optimpliesequi}:
 \begin{theorem}
 \label{Beltranthm}
 If $\left\{\omega_N\right\}_{N=2}^\infty$ is a sequence of well-separated configurations whose spherical cap discrepancy satisfies 
 \[\lim_{N\rightarrow\infty}D_C(\omega_N)\log N = 0,\]
 then  $\left\{\omega_N\right\}_{N=2}^\infty$ are asymptotically optimal with respect to logarithmic energy.

 \end{theorem} 
 \noindent As a corollary, Coulomb points and Fibonacci spiral points are asymptotically log optimal. We make the following natural conjecture:
  
\begin{conjecture}
\label{logequiconj}
  The condition on the spherical cap discrepancy in Theorem \ref{Beltranthm} can be relaxed to
   \[\lim_{N\rightarrow\infty}D(\omega_N) = 0.\]
   Thus a sequence of configurations $\left\{\omega_N\right\}_{N=2}^{\infty}$ is equidistributed and well-separated iff it is asymptotically optimal with respect to logarithmic energy.
\end{conjecture}

\begin{figure}
\centering
\makebox[\textwidth][c]{\includegraphics[width=1.4\textwidth]{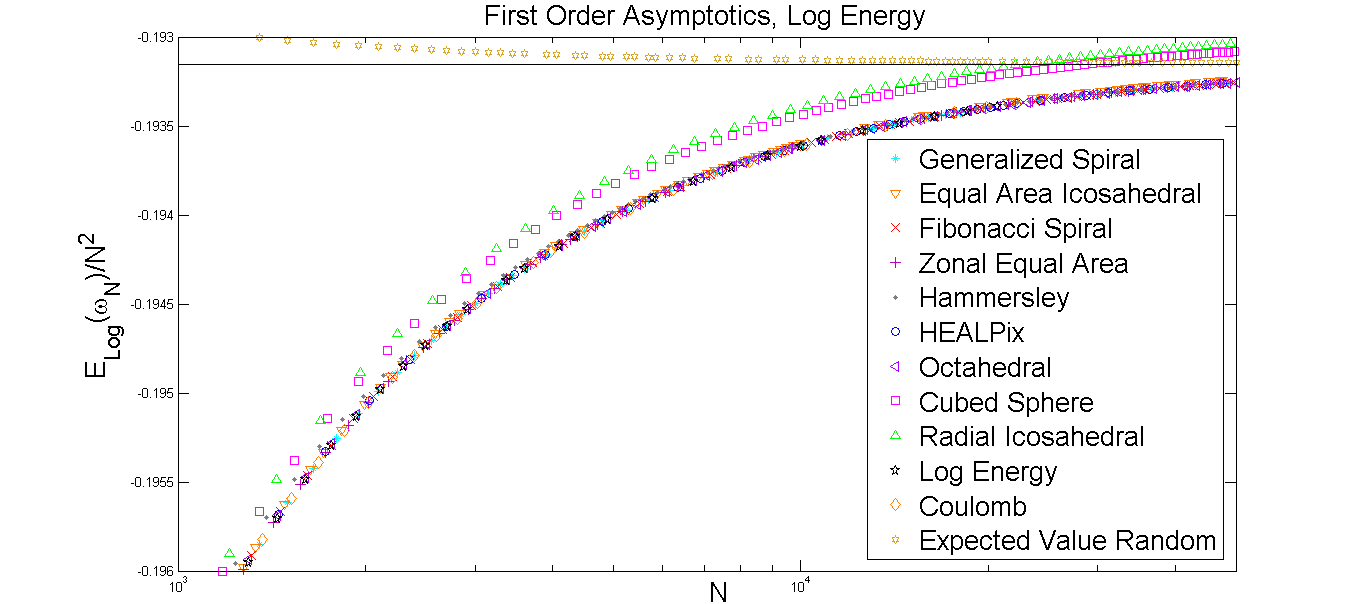}}
\caption{First order asymptotics for log potential. The solid black line is the known first order coefficient in the minimal energy expansion as given in Theorem \ref{logexpansion}.}
\label{logasymptotics1}
\end{figure}

\begin{figure}
\centering
\makebox[\textwidth][c]{\includegraphics[width=1.4\textwidth]{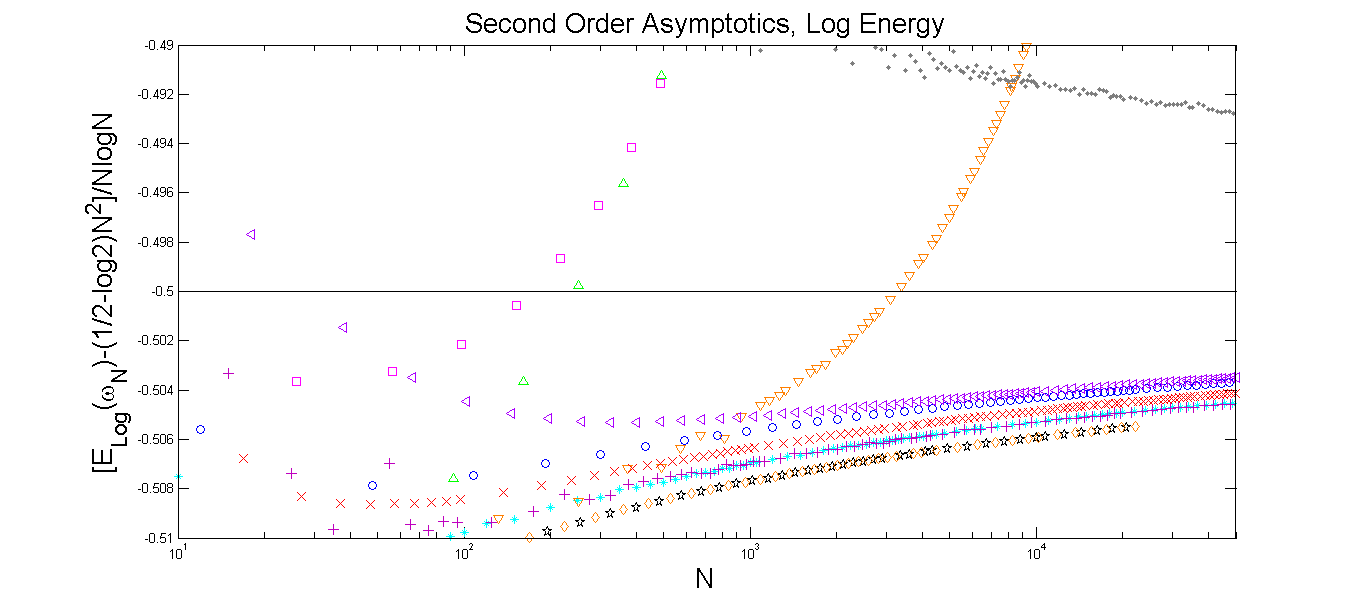}}
\caption{Second order asymptotics for log potential (with the legend of Figure \ref{logasymptotics1}). The solid black line corresponds to the known coefficient from Theorem \ref{logexpansion}.}
\label{logasymptotics2}
\end{figure}

\begin{figure}
\centering
\makebox[\textwidth][c]{\includegraphics[width=1.4\textwidth]{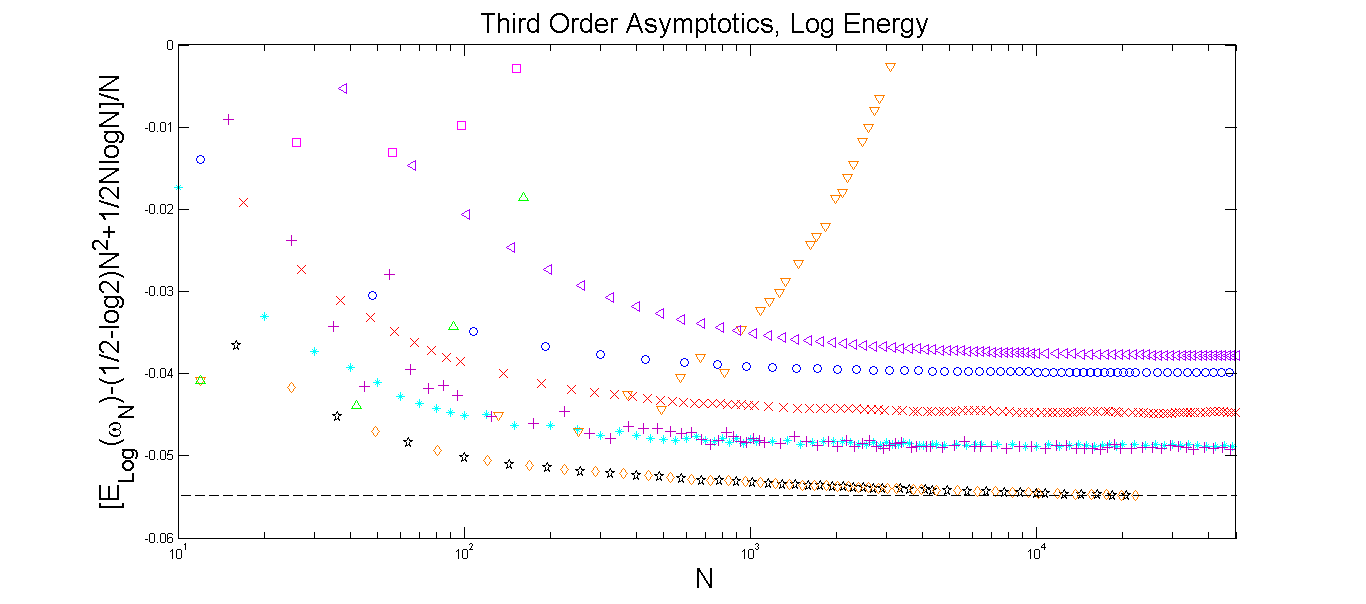}}
\caption{Third order asymptotics for log energy (with the legend of Figure \ref{logasymptotics1}). The dashed line is the conjectured third order constant from \cite{BHS12}.}
\label{logasymptotics3}
\end{figure}

  Figures \ref{logasymptotics1}, \ref{logasymptotics2}, and \ref{logasymptotics3} display three comparisons of the logarithmic energies of the point sets with cardinality up to 50,000.  Some configurations are sampled along subsequences to avoid overcrowding the picture. Due to the computational cost of generating approximate log energy and Coulomb points, these points are only available for $N<22,500$. The energies of random configurations are plotted by expected value. We do not include the maximal determinant nodes because there is no known algorithm to generate them in polynomial time.
  
  Figure \ref{logasymptotics1} shows $E_{\log}(\omega_N)/N^2$ for the point sets. For the equidistributed and well-separated point sets as well as for the Hammersley nodes, this ratio converges to $1/2-\log2$ supporting Conjecture \ref{logequiconj}. Though the radial icosahedral and cubed sphere points are not equidistributed and thus not asymptotically optimal, their log energy appears to be of leading order $N^2$.

 Figure \ref{logasymptotics2} displays
$[E_{\log}(\omega_n) -  \mathcal{I}_{\log} [\sigma] N^2]/N\log N$ for each configuration. The curves for classes of point sets begin separating. The energies of the octahedral, HEALPix, Fibonacci, generalized spiral, zonal equal area, and Coulomb points all appear to converge to the correct second order term. The energy of the Hammersley nodes appears to have $N\log N$ second term, though incorrect coefficient. While the equal area icosahedral nodes perform better than the radial icosahedral nodes, their asymptotic energy appears to have a second term different to $N\log N$. Numerical second order behavior of log energy points is also studied in \cite{NBK}.

Figure \ref{logasymptotics3} compares the energy of the configurations to the conjectured minimal log energy third order term, i.e. $[E_{\log}(\omega_n) - \mathcal{I}_{\log} [\sigma]N^2 + 1/2N\log N]/N$. Conjecture \ref{logconjecture} is supported by the behavior of the Coulomb and log points. The octahedral, HEALPix, Fibonacci, zonal equal area, and generalized spiral configurations appear to have the third term of their energy of order $N$ but the wrong coefficient. Of the algorithmically generated point sets, the generalized spiral and zonal equal area points perform the best with respect to the logarithmic energy.

\
  
\subsection{Riesz Potential, -2 $<\textbf{\textit{s}}<$ 2, $\textbf{\textit{s}}\neq$ 0} 

\

  In this range, the best known bounds on $\mathcal{E}_s(N)$ are due to Wagner (cf. \cite{WagDn} and \cite{WagUp}).
 
 \begin{theorem}
 \label{s1theorem}
 For $-2<s<2$, $s\neq 0$, there exist $c_s,C_s>0$ such that for all $N\in\mathbb{N}$,
 \[c_sN^{1+s/2}\leq \mathcal{E}_s(N) - \mathcal{I}_s[\sigma]N^2\leq C_sN^{1+s/2}.\]
 \end{theorem}
 \noindent Kuiljaars and Saff \cite{KS98} conjecture that $\lim_{N\rightarrow\infty} [\mathcal{E}_s(N) - \mathcal{I}_s[{\sigma}]N^2]/N^{1+s/2}$ exists.
  
  \begin{conjecture}
   For $-2<s<2$, $s\neq 0$,
  \[\mathcal{E}_s(N) = \mathcal{I}_s[{\sigma}]N^2 + \frac{(\sqrt{3}/2)^{s/2}\zeta_{\Lambda_2}(s)}{(4\pi)^{s/2}}N^{1+s/2} + o(N^{1+s/2}),\ \ \ \ \ \ N\rightarrow\infty.\]
  \label{s1conj}
  \end{conjecture}
  Heuristically, the second order coefficient corresponds to the Voronoi decomposition of minimal energy points being composed primarily of close to regular hexagons. A characterization of asymptotically optimal point sets is due to Leopardi \cite{s1char}.
  
  \begin{theorem}
  \label{s1char}
  If a well-separated sequence of configurations $\left\{\omega_N\right\}_{N=2}^\infty$ is equidistributed, then it is asymptotically optimal for $0<s<2$,
  \end{theorem}
      
  \noindent A similar result for $s<0$ is also known (see for example \cite{MEBook}).
  
\begin{theorem}
A sequence of configurations $\left\{\omega_N\right\}_{N=2}^{\infty}$ is asymptotically optimal for $-2<s<0$ iff it is equidistributed.
\label{sneg1char}
\end{theorem}
\noindent Theorems \ref{s1char} and \ref{sneg1char} are the analogs of Conjecture \ref{logequiconj} for $-2<s<2$, $s\neq 0$.

    In the particular case of $s=-1$, the problem of minimizing energy becomes that of maximizing sums of distances. A characterization of minimal configurations for this case is due to Stolarsky \cite{Stolarsky} (see also \cite{BDStol}) and is generalized by Brauchart and Dick in \cite{StoGen}. Let $V_t(\textbf{z}) = \left\{\textbf{x}\ |\ \textbf{x}\cdot\textbf{z}\geq t\right\}$ denote the spherical cap of height $1-t$ centered at $\textbf{z}\in\mathbb{S}^2$. The \textit{$L_2$ discrepancy} of a configuration $\omega_N$ is defined to be
        
        \[D_{L_2}(\omega_N):= \bigg(\int_{-1}^{1}\int_{\mathbb{S}^2}\bigg(\frac{|\omega_N\cap V_t(\textbf{z})|}{N}-\sigma(V_t(\textbf{z}))\bigg)^2 d\sigma(\textbf{z})dt \bigg)^{1/2}.\]
        
        \begin{theorem}
        \textbf{Stolarsky Invariance Principle.} For any $\omega_N\subset\mathbb{S}^2$,
        \[\frac{1}{N^2}\sum_{i\neq j}|\textbf{x}_i-\textbf{x}_j| + 4D_{L_2}(\omega_N)^2 = \mathcal{I}_{-1}[\sigma] = 4/3.\]
        \label{Stolarsky Invariance}
        \end{theorem}
        
        \noindent This gives the immediate corollary (see also \cite{BrauStol}):
        \begin{corollary}
        For a configuration $\omega_N$, $E_{-1}(\omega_N) = \mathcal{E}_{-1}(N)$ iff $\omega_N$ minimizes $L_2$ discrepancy.
        \end{corollary}
  
  Figure \ref{s1asymptotics1} graphs the Coulomb energy of the point sets to first order. From Theorem \ref{s1char}, $\lim_{N\rightarrow\infty} E_1(\omega_N)/N^2 = \mathcal{I}_s[\sigma] = 1$ for all the point sets except the radial icosahedral, cubed sphere, and Hammersley nodes. As in the logarithmic case, the energies of the radial icosahedral and cubed sphere nodes appear to have $N^2$ leading order term but incorrect coefficient. The plot is inconclusive on whether the behavior of the Hammersley nodes is asymptotically optimal. Figure \ref{s1asymptotics2} shows the second order $N^{3/2}$ term from Conjecture \ref{s1conj}. The energies of the generalized spiral, Fibonacci spiral, zonal equal area, HEALPix, and octahedral nodes all appear to converge to an $N^{3/2}$ term that is different from the conjectured value. None of these point sets have a regular hexagonal lattice structure as $N$ gets large. Again the energies of equal area icosahedral points have a different second order term. The Hammersley nodes don't have low enough energy to appear on the plot. The asymptotic energy of the log nodes seems to have the same second order coefficient as the Coulomb nodes and both point sets support Conjecture \ref{s1conj}. As in the logarithmic energy case, the generalized spiral and zonal equal area points perform the best of the algorithmically generated points.
  
       \begin{figure}
       \centering
       \makebox[\textwidth][c]{\includegraphics[width=1.4\textwidth]{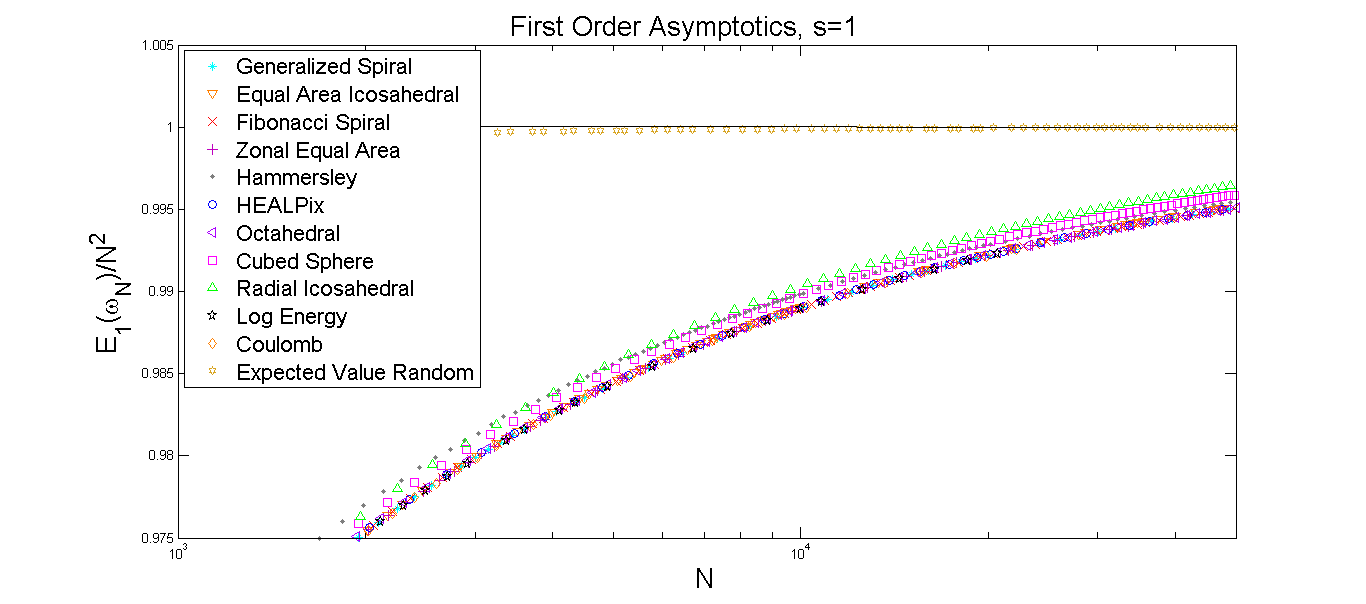}}
       \caption{First order asymptotics for $s = 1$. The solid black line is the known coefficient from Theorem \ref{s1theorem}.}
         \label{s1asymptotics1}
       \end{figure}
       
      \begin{figure}
       \centering
       \makebox[\textwidth][c]{\includegraphics[width=1.4\textwidth]{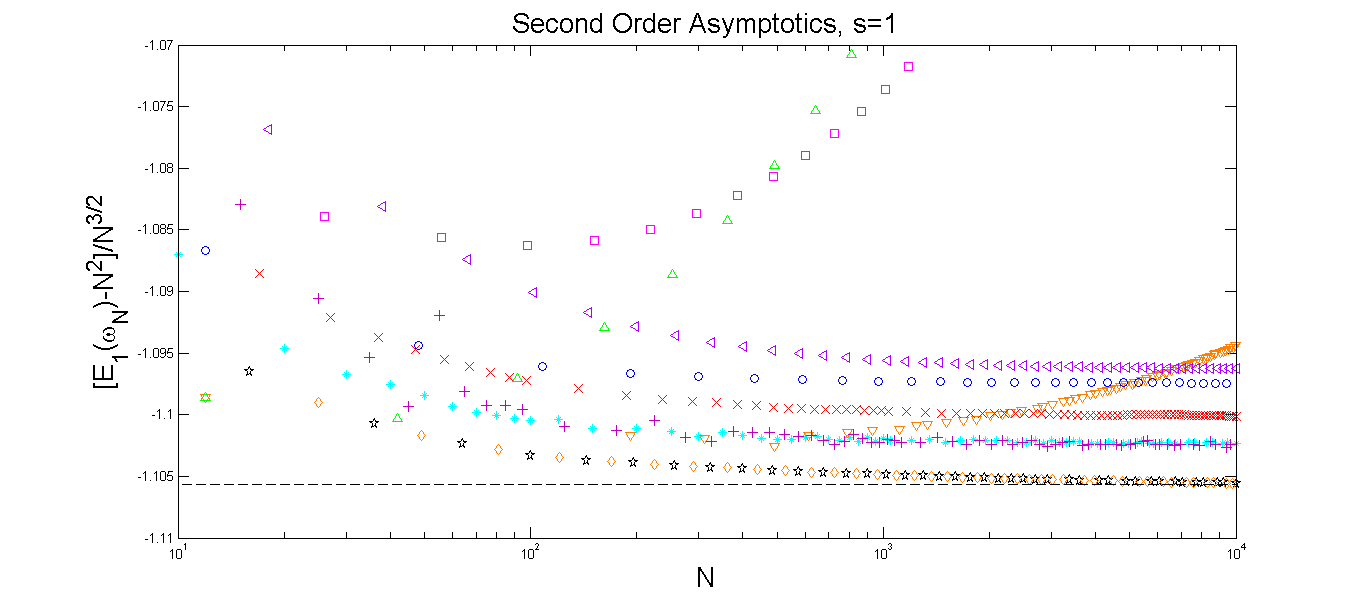}}
       \caption{Second order asymptotics for $s = 1$ (with the legend of Figure \ref{s1asymptotics1}). The dashed line is the conjectured value from \cite{KS98}.}
         \label{s1asymptotics2}
       \end{figure}
    
     \begin{figure}
      \centering
      \makebox[\textwidth][c]{\includegraphics[width=1.4\textwidth]{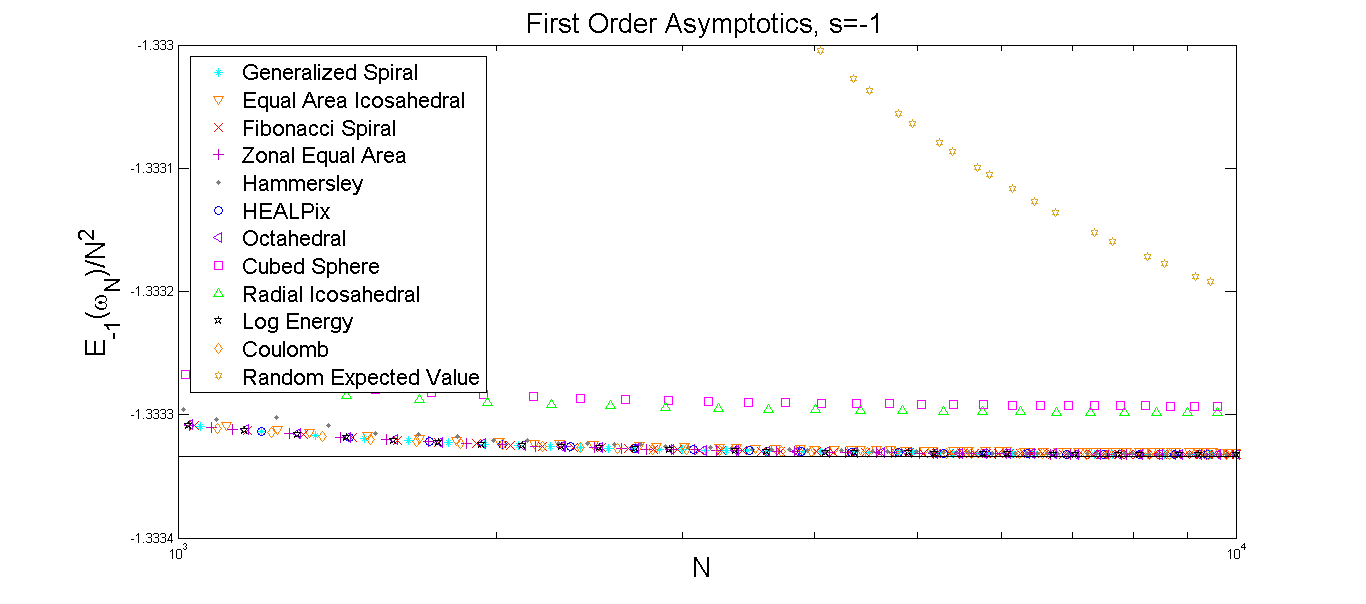}}
      \caption{First order asymptotics for $s = -1$. The solid line follows from the Stolarsky invariance principle.}
        \label{sneg1asymptotics1}
      \end{figure}
      
       \begin{figure}
        \centering
        \makebox[\textwidth][c]{\includegraphics[width=1.4\textwidth]{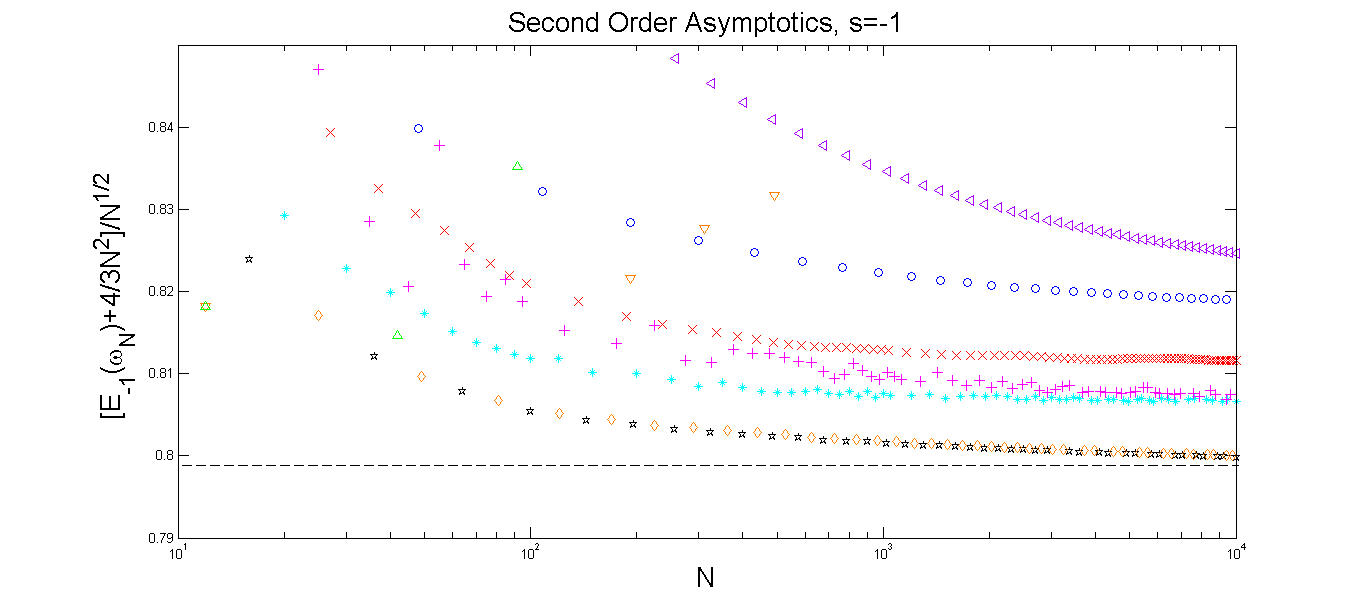}}
        \caption{Second order asymptotics for $s = -1$ (with the legend of Figure \ref{sneg1asymptotics1}). The dashed line is from conjecture \ref{s1conj}.}
          \label{sneg1asymptotics2}
        \end{figure}
    
    Figure \ref{sneg1asymptotics1} displays $E_{-1}(\omega_N)/N^2$ for the point sets. Again the radial icosahedral and cubed sphere points have the correct order though they are not asymptotically optimal. For $s=-1$ The second term in the minimal energy expansion is of known order $\sqrt{N}$ with conjectured coefficient $\approx 0.7985$. As shown in Figure \ref{sneg1asymptotics2}, the behavior of the second order of the asymptotic energy of the point sets resembles that of their log and Coulomb energies. The minimal Coulomb and log energy points perform the best and support Conjecture \ref{s1conj}. The asymptotic energies of the octahedral, HEALPix, Fibonacci, zonal equal area, and generalized spiral appear to converge to the right second order and incorrect coefficient, and the latter two configurations have the lowest energy.

  \subsection{Riesz Potential, $\textbf{\textit{s=2}}$}

\

  Less is known about the case $s=2$. The first order term was proved in \cite{KS98}.
  
  \begin{theorem}
  \[\lim_{N\rightarrow\infty}\frac{\mathcal{E}_2(N)}{N^2\log N} = \frac{1}{4}.\]
  \label{s2expansion}
  \end{theorem}
  
  It is not known whether the $N^2\log N$ term corresponds to any geometric property of the points. There is also no result analogous to Theorems \ref{Beltranthm} and \ref{s1char} giving a sufficient condition for a sequence of configurations to be asymptotically optimal.
  
  Figure \ref{s2asymptotics1} shows $E_2(\omega_N)/N^2\log N$. The Hammersley points do not appear to have asymptotically minimal energy. The behavior of the octahedral, HEALPix, Fibonacci, generalized spiral, zonal equal area, and equal area icosahedral points is not clear. Their energies remain well below the known coefficient but have larger energy than the Coulomb and log energy points. Calculations beyond $N=50,000$ are needed. The generalized spiral, zonal equal area, and equal area icosahedral points perform the best.
  
 \begin{figure}
  \centering
  \makebox[\textwidth][c]{\includegraphics[width=1.4\textwidth]{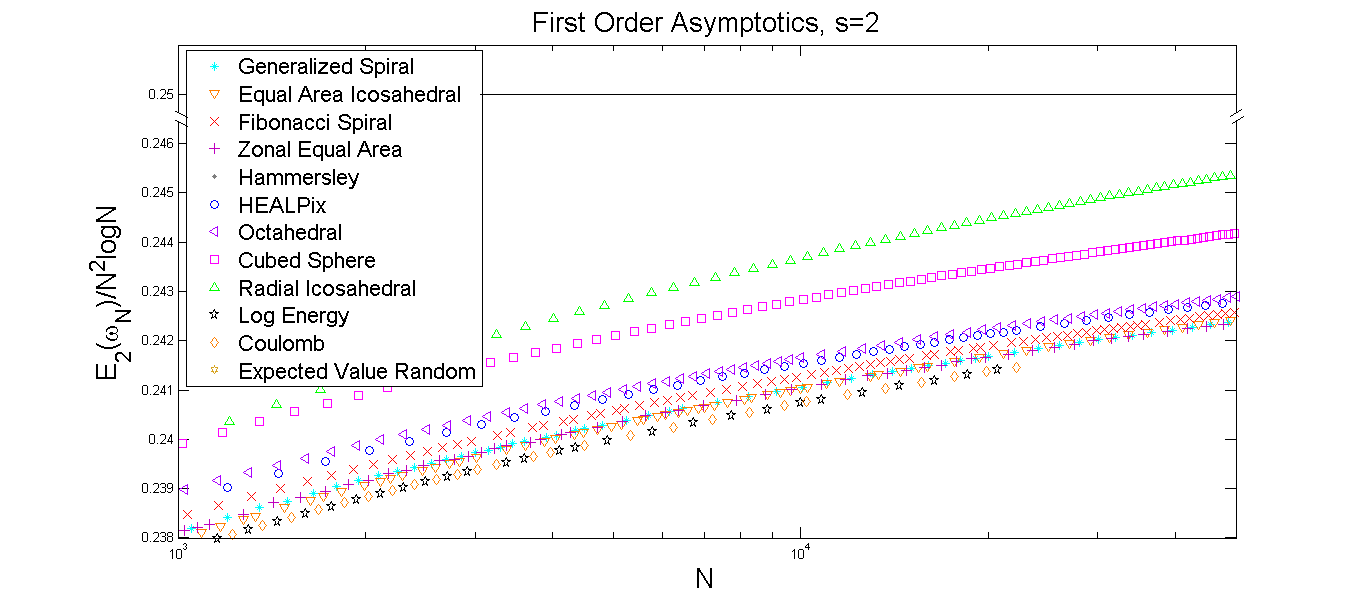}}
  \caption{First order asymptotics for $s = 2$. The break in the y-axis values is to enhance the separation between the configurations. The solid line is the known coefficient from Theorem \ref{s2expansion}.}
    \label{s2asymptotics1}
  \end{figure}  
  
  \
  
  \subsection{Riesz Potential, $\textbf{\textit{s}}>2$}

 \
  
      \begin{figure}
      \centering
      \makebox[\textwidth][c]{\includegraphics[width=1.4\textwidth]{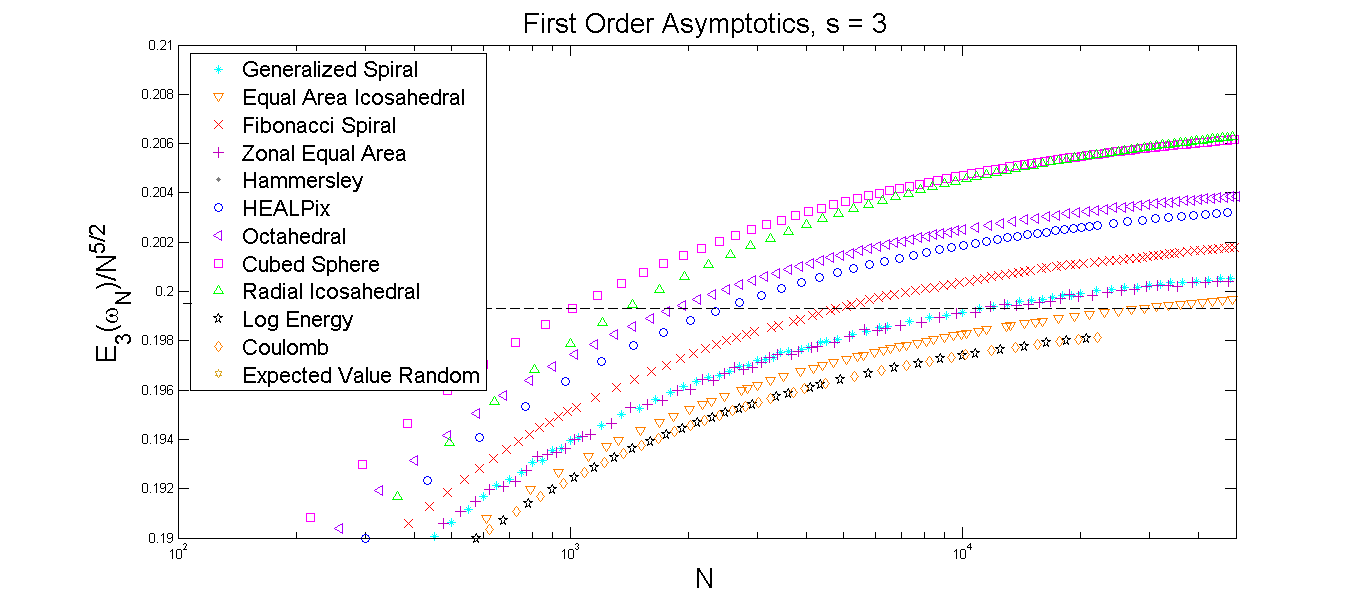}}
      \caption{First order asymptotics for $s = 3$. The dashed line is the conjectured coefficient.}
        \label{s3asymptotics1}
      \end{figure}
  
  In the hypersingular case $s>2$ the continuous energy integral is infinite for any probability measure $\mu$. Instead, the energy is known to be dominated by the nearest neighbor interactions of points as shown in a much more general result by Hardin and Saff \cite{HS2005}.
  
  \begin{theorem}
  For $s>2$, there are constants $C_s$ such that
  \[\lim_{N\rightarrow\infty}\frac{\mathcal{E}_s(N)}{N^{1+s/2}} = C_s,\ \ \ \ \ \ \ \ 0<C_s<\infty.\]
  \label{hypersingular}
  \end{theorem}
  \noindent One can consider the analytic continuation of $\mathcal{I}_s[\sigma]$, $0<s<2$, 
  
  \[V_s: = \frac{2^{1-s}}{2-s},\ \ \ \ s\in\mathbb{C}\setminus \left\{2\right\}.\]
  The following is analogous to Conjecture \ref{s1conj} (see \cite{MEBook}).
  
  \begin{conjecture}
  For $2<s<4$
    \[\mathcal{E}_s(N) = \frac{(\sqrt{3}/2)^{s/2}\zeta_{\Lambda_2}(s)}{(4\pi)^{s/2}}N^{1+s/2}+V_sN^2 + o(N^2).\]
    \label{s3conj}
  \end{conjecture}

    Figure \ref{s3asymptotics1} plots $E_3(\omega_N)/N^{5/2}$. The energies of most configurations seem to be going to the correct order but incorrect coefficient. The equal area icosahedral points outperform the spiral and zonal equal area points of the algorithmically generated configurations. This is expected because their Voronoi decomposition is closest to the regular hexagonal lattice. The log energy and Coulomb points again seem to be close to minimal and may converge to the conjectured $C_3\approx 0.199522$. The Hammersley points are not seen on the plot because their asymptotic energy does not appear to have first order $N^{5/2}$.
  
We conclude this section with the following observation: The apparent first and second order asymptotically minimal behavior of the energy of the Coulomb and log points for the potentials we have considered is striking and we suggest a conjecture.

\begin{conjecture}
The minimal logarithmic energy points and Coulomb points are asymptotically minimal to second term precision for all $-2<s<4$, $s\neq 0$. I.e., combined with Conjectures \ref{s1conj} and \ref{s3conj}, letting $\omega_N:=\omega_N^{log}$ or $\omega_N^{Coul}$

\[\lim_{N\rightarrow\infty}\frac{E_s(\omega_N)-\mathcal{I}_s[\sigma]N^2}{\mathcal{E}_s(N)-\mathcal{I}_s[\sigma]N^2} = 1\ \ \ or\ \ \ \lim_{N\rightarrow\infty}\frac{E_s(\omega_N)-\frac{(\sqrt{3}/2)^{s/2}\zeta_{\Lambda_2}(s)}{(4\pi)^{s/2}}N^{1+s/2}}{\mathcal{E}_s(N)-\frac{(\sqrt{3}/2)^{s/2}\zeta_{\Lambda_2}(s)}{(4\pi)^{s/2}}N^{1+s/2}}=1.\]

Furthermore, the Coulomb points are minimal in the logarithmic energy to the third order term. That is

\[\lim_{N\to\infty}\frac{E_{\log}(\omega_N^{Coul})-(1/2-\log2)N^2+1/2N\log N}{\mathcal{E}_{\log}(N)-(1/2-\log2)N^2+1/2N\log N}=1.\]
\label{newconj}
\end{conjecture}

\section{Proofs}

We begin with two auxiliary results.

\begin{proposition}
Let $\left\{P_N\right\}_{N=1}^\infty$ be a diameter bounded sequence of asymptotically equal area partitions of $\mathbb{S}^2$ such that each $P_N$ has $N$ cells. For each $P_N$, let $\omega_N$ be a configuration of points on $\mathbb{S}^2$ such that the interior of each cell of $P_N$ contains exactly one point of $\omega_N$. Then $\left\{\omega_N\right\}_{N=1}^\infty$ is equidistributed and provides a covering of $\mathbb{S}^2$ with $\eta(\omega_N) \leq CN^{-1/2}$ for all $N\in \mathbb{N}$, where $C$ is as in equation \textup{(\ref{diambdd})}.

\label{lemmaprop}
\end{proposition}

\begin{proof}
The bound on the covering radius is trivial. Let $A\subset \mathbb{S}^2$ be a spherical cap and let 
\[A_\delta := \left\{x\in\mathbb{S}^2:\ \textup{dist}(x,A)\leq \delta \right\},\] where dist$(x,A) := \min_{y\in A}|x-y|$ is the standard distance function. For $x\in\omega_N$, denote by $W_x^N$ the cell of $P_N$ containing $x$. Let $\epsilon,\delta>0$ and choose $N$ large enough such that 
\[N\min_x\sigma(W_x^N)\geq 1-\epsilon\]
and $x\in A\cap\omega_N$ implies $W_x^N\subset A_\delta$. Then

\[\frac{|\omega_N\cap A|}{N}\leq \frac{|\left\{x\ :\ W_x^N\subset A_\delta\right\}|}{N}\leq\frac{\sigma(A_\delta)}{N\min_x\sigma(W_x^N)}\leq \frac{\sigma(A_\delta)}{1-\epsilon}.\]
Since $\epsilon$ is arbitrary, we have
\[\limsup_N\frac{|\omega_N\cap A|}{N}\leq \sigma(A_\delta).\]
Letting $\delta\to 0$ gives
\begin{equation}
\limsup_N\frac{|\omega_N\cap A|}{N}\leq \sigma(A).
\label{lemmaeq}
\end{equation}
Applying inequality (\ref{lemmaeq}) to $\mathbb{S}^2\setminus A$, we obtain
\[\liminf_N\frac{|\omega_N\cap A|}{N}=1-\limsup_N\frac{|\omega_N\cap(\mathbb{S}^2\setminus A)|}{N}\geq 1-\sigma(\mathbb{S}^2\setminus A) = \sigma(A),\]
and thus, we have
\[\lim_{N\to\infty}\frac{|\omega_N\cap A|}{N}=\sigma(A).\]
\end{proof}

\

\noindent\textit{Proof of Theorem \ref{spiral distribution}.} 

For a fixed $N$ denote in spherical coordinates $x_i:=(\phi_i,\theta_i)\in\omega_N$. We first prove the separation bound. Let $\epsilon_N = 2\sqrt{4\pi/N}$. For large $N$, by the pigeonhole principle, at least one set of adjacent points $x_k,x_{k+1}\in\omega_N$ is separated along $S_N$ by a distance less than $\sqrt{4\pi/N}$. So we can restrict our attention to $k\leq N/2$ and $\epsilon_N$-balls $B(x_k,\epsilon_N)$. For large N, if $k<8\pi+1/2$, then
\[\cos\phi_k = 1-\frac{2k-1}{N}\geq 1-\frac{16\pi}{N} \approx \cos(2\sqrt{\frac{4\pi}{N}}),\]
and $x_k$ is within the first two full longitudinal turns of $S_N$ starting from a pole. Otherwise, $B(x_k,\epsilon_N)$ contains disjoint levels of $S_N$. In this case, the minimal distance between levels in $B(x_k,\epsilon_N)$ is $\sqrt{4\pi/N}$. We compute the nearest neighbor distance between points in the same level as follows. Let

\[f_N(k):=\sqrt{N}|x_k-x_{k+1}|, \ \ \ \ x_k\in \omega_N,\ \ \ \ k<N/2.\]
Using the distance formula for spherical coordinates
\begin{equation}
|x_j-x_k|^2 = 2-2(\cos\phi_j\cos\phi_k+\sin\phi_j\sin\phi_k\cos(\theta_j-\theta_k))
\label{geodesic}
\end{equation}
and expanding $\cos^{-1}x$ around $x=1$ we have
\begin{equation}
\lim_{N\to\infty}f_N(k) = \sqrt{8k-4\sqrt{4k^2-1}\cos(\sqrt{2\pi}(\sqrt{2k-1}-\sqrt{2k+1}))},
\label{spirallimit}
\end{equation}
and $f_N(k)<f_{N+1}(k)$ for all $k<N/2-1$ and large N. Thus we have the correct order for the minimal separation between adjacent points in $\omega_N$. Furthermore, (\ref{spirallimit}) is increasing as a function of $k$ and thus
\begin{equation}
\lim_{N\to\infty} \min_k f_N(k) = \sqrt{8-4\sqrt{3}\cos(\sqrt{2\pi}(1-\sqrt{3}))}.
\label{spiralmin}
\end{equation}
Lastly, around the north pole, $k_1,k_2<8\pi+1/2$, we can again use (\ref{geodesic}) to show that
\[\lim_{N\to\infty} \sqrt{N}|x_{k_1}-x_{k_2}|\]
exists and can be computed case by case for pairs $(k_1,k_2)$. The southern hemisphere can be similarly computed. Comparing to (\ref{spiralmin}) gives the separation constant (\ref{spiralsep}).

  \begin{figure}[h!]
  \centering
  \includegraphics{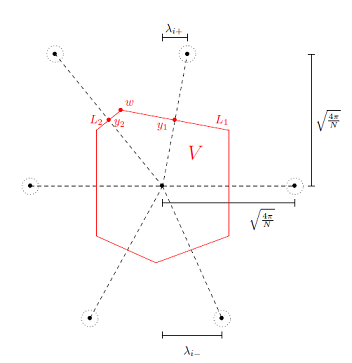}
  \caption{Limiting form of Voronoi cell for each $x_i\in K_h\cap\omega_N$ with relevant quantities labeled. We can choose $N$ large enough such that the nearest neighbors of $x_i$ lie within $\epsilon$ of the points taken along iso-latitudinal lines separated by $\sqrt{4\pi/N}$}
  \label{spiralvoronoicell}
  \end{figure}
    
    For covering, given $y\in\mathbb{S}^2$,
    \[\textup{dist}(S_N)\leq \sqrt{\frac{\pi}{N}}.\]
 From (\ref{spirallimit}), the maximal distance from any point on $S_N$ to a point of $\omega_N$ is $O(1/\sqrt{N})$ and thus the covering radius of $\omega_N$ is also $O(1/\sqrt{N})$.
 
 We have two additional observations. First, the Voronoi decompositions of the spiral points are diameter bounded. Secondly, the Voronoi cells are asymptotically equal area on $K_h := \left\{(x,y,z)\in \mathbb{S}^2:-h\leq z\leq h \right\}$ for any fixed $0<h<1$. By this we mean

 \[\lim_{N\rightarrow\infty}N\max_{V_x(\omega_N)\subset K_h} \sigma(V_x(\omega_N)) = \lim_{N\rightarrow\infty}N\min_{V_x(\omega_N)\subset K_h} \sigma(V_x(\omega_N)) = 1.\]

 Indeed, fixing $h$, for any $\epsilon>0$, we can take $N$ large enough such that given $x_i=(\phi_i,\theta_i)\in K_h\cap \omega_N$, $x_{i-1}$, and $x_{i+1}$ are almost iso-latitudinal with $x_i$ with separation $\sqrt{4\pi/N}$. I.e., 
 \[|x_{i\pm 1}-(\phi_i,\theta_i\pm\csc\phi_i\sqrt{4\pi/N})|<\epsilon.\]
 There exists shifts $0\leq\lambda_{i+},\lambda_{i-}\leq\sqrt{4\pi/N}$ such that the nearest points in the adjacent spiral levels are within $\epsilon$ of the points $(\phi_i\pm\sqrt{4\pi/N},\theta_i+\csc\phi_i\lambda_{i\pm})$ and $(\phi_i\pm\sqrt{4\pi/N},\theta_i+\csc\phi_i(\lambda_{i\pm}-\sqrt{4\pi/N}))$. Thus as $N\to\infty$, the Voronoi cell $V_i(\omega_N)$ approaches the form of $V$ in Figure \ref{spiralvoronoicell} and 
 \[\sigma(V_i(\omega_N)) = \sigma(V) + O(\epsilon^2).\]
 Furthermore, we can treat $V$ as a planar polygon in Figure \ref{spiralvoronoicell}, and 
 \[\sigma(V)=\frac{1}{N}\]
 independent of the shifts $\lambda_{i\pm}$. This we show by direct computation.
 
   Letting $a=\sqrt{4\pi/N}$ and centering $x_i$ at $(0,0)$, the points $y_1 = (\lambda_{i+}/2,a/2)$ and $y_2 = ((\lambda_{i+}-a)/2,a/2) $ are the midpoints of the lines connecting $x_i$ to its nearest neighbors in the adjacent level which are shifted by $\lambda_{i+}$. The corresponding lines
  \[L_1: y = -\frac{\lambda_{i+}}{a}(x-\frac{\lambda_{i+}}{2})+\frac{a}{2} \ \ \ \ \ \ \ L_2: y = \frac{a-\lambda_{i+}}{a}(x-\frac{\lambda_{i+}-a}{2})+\frac{a}{2}\]
  form the top boundary of V and have intersection point 
  \[w:= (\lambda_{i+}-a/2,(\lambda_{i+}a+a^2-\lambda_{i+}^2)/2a).\]
  From this we calculate the area of the top half of V to be
  \[\sigma(V_{top}) = \frac{1}{4\pi}\bigg[a\bigg(\frac{-a\lambda_{i+}+\lambda_{i+}^2+a^2}{2a}\bigg)+\frac{a}{2}\bigg(\frac{a\lambda_{i+}-\lambda_{i+}^2}{a}\bigg)\bigg]=\frac{1}{2N}.\]
The same calculation holds for the bottom half of $V$ and thus (\ref{asympeq}) holds.
  
We now consider equidistribution. Because the height steps between points in $\omega_N$ are uniform, for a spherical cap $A$ centered at a pole,
\begin{equation}
\lim_{N\to\infty}\frac{|\omega_N\cap A|}{N}=\sigma(A).
\label{equicap}
\end{equation}
If $A$ does not contain one of the poles, then $A\subset K_h$ for some $h$ and (\ref{equicap}) holds by Proposition \ref{lemmaprop}. Finally if $A$ is a cap containing but not centered at one of the poles, $A$ can be partitioned into an open cap of height $h$ centered at the pole and $A\cap K_h$. Because (\ref{equicap}) holds on each disjoint subset, it also holds on~$A$.
\qed

\

\noindent\textit{Proof of Proposition \ref{Fib mesh ratio}.}

When  $z=k,$ the basis vector $\textbf{c}_{k,i}$ has minimum length of $\sqrt{2}d$. At this latitude, $\textbf{c}_{k+1,i}$ and $\textbf{c}_{k-1,i}$ form the next most dominant spirals and have length $\sqrt{3}d$. For $z = k + 1/2$, $\textbf{c}_{k,i}$ and $\textbf{c}_{k+1,i}$ are equally dominant and have grid length $\sqrt[4]{5}d$. For a fixed latitude $\phi\neq \pm\pi/2$, $z$ increases with $N$. Points around $\phi$ will form a locally rectangular grid. Thus the separation approaches $\sqrt{2}d$ which occurs when $z=k$ and the largest hole in the triangulation around $x_i$ will be at most $\sqrt[4]{5}d/\sqrt{2}$ which occurs when $z=k\pm 1/2$. Thus off the poles, $\delta(\omega_{2N+1})\geq \sqrt{8\pi\sqrt{5}}/\sqrt{N}$ and $\eta(\omega_{2N+1})\leq \sqrt{2\pi/N}.$

Indeed, these inequalities hold for large zone numbers $z$ where $F_k\approx \varphi^k/\sqrt{5}$ holds. However, on the polar points, $x_1$ and $x_N$, the zone number
\[z=\frac{\log((2N+1)\pi\sqrt{5}(1-4N^2/(2N+1)^2))}{\log\varphi^2}\to\frac{\log(2\pi\sqrt{5})}{\log\varphi^2} = 2.75...,\ \ \ \ N\to\infty\]
and writing $\textbf{c}_{k,i}$ in the form of equation (\ref{Fib basis2}) for small $k$ overestimates the length of the vector. Using equation (\ref{Fib basis}) and noticing that $k-1/2\leq z \leq k+1/2$ implies
\[\frac{\varphi^{2k-1}}{(2N+1)\pi\sqrt{5}}\leq\cos^2\phi\leq\frac{\varphi^{2k+1}}{(2N+1)\pi\sqrt{5}},\]
we have
\[|\textbf{c}_{k,i}|^2\geq \frac{4\pi^2\varphi+20\pi F_k^2\varphi^{-2k+1}}{(2N+1)\pi\sqrt{5}}=O\bigg(\frac{1}{2N+1}\bigg), \ \ \ \ \ k-1/2\leq z \leq k+1/2.\]
Since $|\textbf{c}_{k,i}|$ is the minimal separation distance for $k-1/2\leq z\leq k+1/2$, we have the correct order of separation on $\mathbb{S}^2$. By a similar computation, we have the upper bound
\[|\textbf{c}_{k,i}|\leq O\bigg(\frac{1}{\sqrt{2N+1}}\bigg),  \ \ \ \ \ \ \ \ k-3/2\leq z \leq k+3/2.\]
Since the triangulation of $\omega_{2N+1}$ in each zone consists of $\textbf{c}_{k,i}$, $\textbf{c}_{k-1,i}$, and $\textbf{c}_{k+1,i}$, the covering of $\omega_{2N+1}$ is of the correct order on all of $\mathbb{S}^2$.
\qed

\

\noindent\textit{Proof of Proposition \ref{zonal quasi}.}

The diameter boundedness of the partition and Proposition \ref{lemmaprop} gives equidistribution and covering,

\[\eta(\omega_N) \leq \frac{3.5}{\sqrt{N}}.\]
In \cite{LeopardiDiam} and \cite{Zhou}, it is established that there exists $c_1,c_2>0$ such that for all partitions $P_N$, with $\phi_j$, $n$ as defined above,

\begin{equation}
\frac{c_1}{\sqrt{N}}\leq\phi_{j+1}-\phi_j\leq \frac{c_2}{\sqrt{N}},\ \ \ \ \ \ \ \ \ 0\leq j \leq n.
\label{zonal collar order}
\end{equation}

\noindent This gives the correct order of separation between collars. For neighbors $x_1,x_2 \in \omega_N$ within collar $j$, wlog suppose $\phi_j<\pi/2$. Using the fact that the normalized area of each cell can be expressed as

\[\frac{(\cos\phi_j-\cos\phi_{j+1})}{2y_i} = \frac{1}{N},\]
we have

\[|x_1-x_2|\geq \frac{2\pi\sin\phi_j}{y_i} = \frac{4\pi\sin\phi_j}{N(\cos\phi_j-\cos\phi_{j+1})}.\]
So it suffices to show there exists $c_3>0$ such that

\begin{equation}
\frac{\sin\phi_j}{\cos\phi_j-\cos\phi_{j+1}}\geq c_3\sqrt{N}\ \ \ \ \ \ \forall N, 0\leq j\leq n.
\label{zonalsep}
\end{equation}
For a fixed $h>0$ and $\phi_j\geq h$, this follows from (\ref{zonal collar order}) and the fact that $\cos$ is Lipschitz.
On the other hand, for sufficiently small $\phi_j$, there exists $c_4>0$ such that

\[\frac{\sin\phi_j}{\cos\phi_j-\cos\phi_{j+1}}\geq c_4\frac{\phi_j}{\phi_j^2-\phi_{j+1}^2}.\]
Again applying (\ref{zonal collar order}) twice,

\[\textup{RHS}\ (\ref{zonalsep})\geq \frac{c_4}{c_1}\frac{\phi_j}{\phi_j+\phi_{j+1}}\sqrt{N}\geq \frac{c_4}{c_1}\frac{\phi_j}{2\phi_j+\frac{c_2}{\sqrt{N}}}\sqrt{N}\geq c_3\sqrt{N}.\]
In the last step we used the fact that for some $c_5>0$ and all $j$

\[\phi_j\geq\cos^{-1}(1-\frac{2}{N})\geq \frac{c_5}{\sqrt{N}}.\]
\qed

\

\noindent\textit{Proof of Proposition \ref{HEALPix quasi}.}

The pixels are diameter bounded and thus by Proposition \ref{lemmaprop} the nodes are equidistributed.

To establish separation, we examine the five cases of nearest neighbor points: The points lie in 1) the polar region or 2) the equatorial region, and the points lie in a) the same ring or b) adjacent rings, and 3) the points lie in adjacent rings at the boundary of the polar region and the equatorial region.

Case 1a: If nearest neighbor points lie in the polar region along the same ring $1\leq i\leq k$ that has radius \[r_i = \sin\phi_i= \sqrt{\frac{2i^2}{3k^2} - \frac{i^4}{9k^4}}\]
 and $4i$ equally spaced points, then the separation $\delta$ satisfies 
 \[\delta = 2r_i\sin\frac{\pi}{4i} = 2\sqrt{\frac{2i^2}{3k^2} - \frac{i^4}{9k^4}}\sin\frac{\pi}{4i}\geq\frac{\sqrt{2}}{i}\sqrt{\frac{2i^2}{3k^2} - \frac{i^4}{9k^4}} = O\bigg(\frac{1}{k}\bigg) = O\bigg(\frac{1}{\sqrt{N}}\bigg).\]
 In the middle inequality we use the fact that
\begin{equation}
\sin x \geq \frac{\sqrt{2}/2}{\pi/4}x\ \ \ \ \ \ \ \ \  0\leq x \leq \pi/4.
\label{sineineq}
\end{equation}

Case 2a: Suppose the nearest neighbor points lie in the equatorial region along the same ring. Since each ring has $4k$ points, the smallest separation occurs at the ring farthest from the equator and closest to $z = 2/3$. Using (\ref{sineineq}) again, we have
\[ \delta = 2r_i\sin\frac{\pi}{4k}\geq \frac{2\sqrt{5}}{3}\sin\frac{\pi}{4k}\geq\frac{\sqrt{10}}{3k} = O\bigg(\frac{1}{\sqrt{N}}\bigg).\]

Case 1b: We split up the rings in the polar region into the outer half, $1\leq i \leq k/2$ and the inner half, $k/2\leq i\leq k$. On the outer half, the separation between rings is 
\begin{align*}
\delta \geq r_{i+1}-r_i =& \frac{(i+1)\sqrt{6k^2-(i+1)^2}-i\sqrt{6k^2-i^2}}{3k^2}\\
\geq& \frac{(k/2+1)\sqrt{6k^2-(k/2+1)^2}-(k/2)\sqrt{6k^2-(k/2)^2}}{3k^2} =  O\bigg(\frac{1}{\sqrt{N}}\bigg).\\
\end{align*}
\noindent On the inner polar rings, the separation between rings is 

\[\delta \geq |\cos\phi_{i+1}-\cos\phi_i| = \frac{2i+1}{3k^2}\geq \frac{k+1}{3k^2} =O\bigg(\frac{1}{\sqrt{N}}\bigg).\]

Case 2b: In the equatorial region, the ring height $z$ increases linearly with respect to the index $i$ giving

\[\delta \geq \frac{2}{3k} =O\bigg(\frac{1}{\sqrt{N}}\bigg).\]

Case 3 follows from Case 1b and 2b.

The covering of the points follows by similar geometric arguments.
\qed

\

\noindent\textit{Proof of Proposition \ref{rad icos quasi}.}

It suffices to show $\Pi$ is locally bi-Lipschitz on each face $\mathcal{F}$ of the icosahedron. If for some $\delta$, $L_1$, $L_2 > 0$

\[L_1|x-y|\leq|\Pi(x)-\Pi(y)|\leq L_2|x-y|, \ \ \ \ \ x,y\in \mathcal{F},\ \  |x-y|<\delta,\]
then

\[\gamma(\omega_N)\leq\frac{L_2}{L_1}\gamma(\widetilde{\omega_{N_k}}) = \frac{L_2}{L_1\sqrt{3}}.\]
Let $c := \min_{x \in \text{Icos}} |x| = \sqrt{(1/3+2\sqrt{5}/15)}$. For $x,y\in \mathcal{F}$ with angle $\theta$, we have

\[|x-y|\geq 2c\sin\frac{\theta}{2}.\]
Using the fact that $\sin^{-1}x\leq \frac{\pi}{2}x$ for $x\leq 1$ and $ |x-y|< 2c$,

\[|\Pi(x)-\Pi(y)|\leq \theta\leq 2 \sin^{-1}\bigg(\frac{|x-y|}{2c}\bigg)\leq\pi c |x-y|.\]
For the other inequality, wlog suppose $c\leq|x|\leq|y|\leq 1$, and consider the line $P \subset \mathcal{F}$ connecting $x$ and $y$. Denote $z$ as the projection of $0$ onto $P$. Defining $\phi := \cos^{-1}(|z|/|x|)$, we have

\[|y|-|x| = |z|\sec(\theta+\phi)-|z|\sec\phi\leq \sec(\theta+\phi)-\sec\phi.\]
Since $\phi \leq \cos^{-1}c-\theta$ and secant is convex on $(0,\pi/2)$,
\[|y|-|x|\leq\frac{1}{c}-\sec(\cos^{-1}c-\theta)=:g(\theta).\]
Thus,
\begin{align*}
|x-y|^2 &= |x|^2+|y|^2-2|x||y|\cos \theta\leq\max_{c\leq|y|\leq 1} (|y|+g)^2+|y|^2-2|y|(|y|+g)\cos\theta\\
& = (1+g)^2+1-2(1+g)\cos\theta.\\
\end{align*}
Since
\[f(\theta):=\frac{(1+g)^2+1-2(1+g)\cos\theta}{2-2\cos\theta}\]
is continuous for $\theta\in (0,\pi/2+\cos^{-1}c)$ and $\lim_{\theta\rightarrow 0} f(\theta)$ exists, there exists $L>0$ such that
\[|x-y|\leq L|\Pi(x)-\Pi(y)|, \ \ \ \ \ \ \ \ \ \theta\in (0,\pi/2+\cos^{-1}c).\]
\qed

\

\noindent\textit{Proof of Theorem \ref{oct quasi}.}
Restricting ourselves to the face of $\mathbb{K}$ with all positive coordinates, label the vertices of the partition $\left \{A_{i,j}\right \}_{0\leq i+j\leq k}$ by

\[A_{i,j} = \bigg( \frac{iL}{k\sqrt{2}}, \frac{jL}{k\sqrt{2}}, \frac{L}{\sqrt{2}}\bigg(1 - \frac{i+j}{k}\bigg)\bigg).\]
Let $\mathcal{A}_{i,j} = \mathcal{U}(A_{i,j})$. Then

\[\mathcal{A}_{i,j} = \bigg( \frac{i+j}{k}\sqrt{2-\frac{(i+j)^2}{k^2}}\cos\frac{\pi j}{2(i+j)}, \frac{i+j}{k}\sqrt{2-\frac{(i+j)^2}{k^2}}\sin\frac{\pi j}{2(i+j)}, 1-\frac{(i+j)^2}{k^2}\bigg).\]
Then
\begin{equation}
\delta (\omega_N) = \min_{i,j}\{ \|\mathcal{A}_{i+1,j}-\mathcal{A}_{i,j}\|, \|\mathcal{A}_{i,j+1} - \mathcal{A}_{i,j}\|, \|\mathcal{A}_{i+1,j} - \mathcal{A}_{i,j+1}\|\}.
\label{OctNeighbors}
\end{equation}
Adapting \cite{Oct}, we have

\[\|\mathcal{A}_{i+1,j} - \mathcal{A}_{i,j}\|^2 = 2\frac{(i+j+1)^2}{k^2}+2\frac{(i+j)^2}{k^2} - 2\frac{(i+j+1)^2(i+j)^2}{k^4} \ \ \ \ \ \ \ \ \ \ \]

\[ \ \ \ \ \ \ \ \ \ \ \ \ \ \ \ \ \ \ \ \ \ \   -2\frac{(i+j)(i+j+1)}{k^2}\sqrt{2-\frac{(i+j+1)^2}{k^2}}\sqrt{2-\frac{(i+j)^2}{k^2}}\cos\frac{\pi j}{2(i+j)(i+j+1)}.\] 
Along the line $i+j = c$, the minimum is obtained when the cosine term is maximized, i.e. at $j=0$. Thus

\[\min_{i,j}\|\mathcal{A}_{i+1,j} - \mathcal{A}_{i,j}\|^2 = \min_{0\leq i\leq k} \|\mathcal{A}_{i+1,0} - \mathcal{A}_{i,0}\|^2\]
\[ = \frac{2}{k^2}\min_{0\leq i \leq k}\bigg((i+1)^2 +i^2+\frac{(i+1)^2i^2}{k^2} - i(i+1)\sqrt{2-\frac{(i+1)^2}{k^2}}\sqrt{2-\frac{i^2}{k^2}}\bigg)= \frac{2}{k^2}.\]
By symmetry of the above expressions in $i$ and $j$,

\[\min_{i,j}\|\mathcal{A}_{i,j+1} - \mathcal{A}_{i,j}\|^2 = \min_{0\leq j \leq k}\|\mathcal{A}_{0,j+1} - \mathcal{A}_{0,j}\|^2 = \frac{2}{k^2}.\]
Lastly,

\[\|\mathcal{A}_{i+1,j} - \mathcal{A}_{i,j+1}\|^2 = 4\frac{(i+j+1)^2}{k^2}\bigg(2 - \frac{(i+j+1)^2}{k^2}\bigg)\sin^2\frac{\pi}{4(i+j+1)}\]
which again depends on only $i+j$. Using (\ref{sineineq}) we have
\begin{align*}
\min_{i,j}\|\mathcal{A}_{i+1,j} - \mathcal{A}_{i,j+1}\|^2 &= \min_{0\leq i \leq k} \|\mathcal{A}_{i+1,0} - \mathcal{A}_{i,1}\|^2\\
&= 4 \min_{0\leq i \leq k} \frac{(i+1)^2}{k^2}\bigg(2 - \frac{(i+1)^2}{k^2}\bigg)\sin^2\frac{\pi}{4(i+1)}\\
&\geq 4\min_{0\leq i\leq k} \frac{(i+1)^2}{k^2}\bigg(2 - \frac{(i+1)^2}{k^2}\bigg)\frac{1}{2(i+1)^2}\\
&= 4\min_{0\leq i \leq k} \frac{1}{2k^2}\bigg(2 - \frac{(i+1)^2}{k^2}\bigg) = \frac{2}{k^2}\bigg(2- \frac{(k+1)^2}{k^2}\bigg).\\
\end{align*}
Thus from (\ref{OctNeighbors})

\[\delta(\omega_{4k^2+2})^2 \geq \frac{2}{k^2}\bigg(2 - \frac{(k+1)^2}{k^2}\bigg).\]
Taking the square root and substituting $N = 4k^2 +2$ gives

\[\liminf_{N\rightarrow\infty}\delta(\omega_N)\sqrt{N} \geq \sqrt{8}.\]
Finally, the diameter bound in \cite{Oct} gives an immediate upper bound for the covering radius from which (\ref{Octmesh}) follows:

\[\eta(\omega_{4k^2+2},\mathbb{S}^2) \leq \sqrt{\frac{4+\pi^2}{8k^2}}.\]
\qed
\section{Matlab Code}

Many thanks to Grady Wright for his help in implementing some of these point sets in Matlab. Code for the Fibonacci, Hammersley, HEALPix, and cubed sphere nodes authored by Wright is available at\\
\noindent \textbf{\textit{https://github.com/gradywright/spherepts}}. The approximate Coulomb, log energy, and maximum determinant points were provided from computations by Rob Womersley. Code for the zonal equal area points is available in Paul Leopardi's Recursive Zonal Equal Area Toolbox from \textbf{\textit{eqsp.sourceforge.net.}} Code for the generalized spiral, octahedral, radial icosahedral, and equal area icosahedral nodes was created by the authors and is available upon request.

 \bibliographystyle{plain}
 \bibliography{SphereConfigurations8Bibliography}

\end{document}